\documentclass[reqno,10pt]{amsart}
\usepackage{amssymb,amsmath,amscd, amsthm,stmaryrd,mathrsfs,tikz,color, amsfonts,  esvect}
\usepackage[dvipsnames]{xcolor}
 \usepackage{autobreak} 
\usepackage{enumitem}
\setenumerate[1]{label=(\arabic*), font=\normalfont} 
\setenumerate[2]{label=(\alph*), font=\normalfont}   
\usepackage{braket}
\usetikzlibrary{patterns}
\usetikzlibrary{er}

\usepackage{genyoungtabtikz}

\usepackage[colorlinks,linkcolor=blue,urlcolor=cyan,citecolor=blue,pagebackref]{hyperref}

\numberwithin{equation}{section}
\allowdisplaybreaks

\usepackage[a4paper]{geometry}

\usepackage{ 
   amsmath}
\usepackage{ 
   autobreak}

\usepackage{tikz-cd}

\newcommand{\parti}{\mathcal P}

\newcommand{\Ree}{\mathrm{Re}}

\theoremstyle{plain}
\newtheorem{Thm}{Theorem}[section]
\newtheorem{Prop}[Thm]{Proposition}
\newtheorem{Lem}[Thm]{Lemma}
\newtheorem{Cor}[Thm]{Corollary}

\newtheorem*{thm*}{Theorem}
\theoremstyle{definition}
\newtheorem{Defff}[Thm]{Definition}
\newtheorem{ex}[Thm]{Example}
\newtheorem{Rem}[Thm]{Remark}

\newcommand{\lam}{\lambda}

\newcommand{\al}{\alpha}

\newcommand{\ep}{\varepsilon}

\newcommand{\Ann}{\mathrm{Ann}}
\newcommand{\eq}{\begin{equation}}
\newcommand{\en}{\end{equation}}
\newcommand{\eeqna}{\end{eqnarray}}
\newcommand{\beqn}[1]{\begin{equation}\label{#1}}
\newcommand{\eeqn}{\end{equation}}

\newcommand{\mc}[1]{\mathcal{#1}}

\renewcommand{\subset}{\subseteq}
\newcommand{\rar}{\rightarrow}
\newcommand{\bil}[2]{\langle{#1},{#2}^{\vee} \rangle }
\newcommand{\hs}{ \mathfrak{h}^*}

\newcommand{\ev}{^{\mathrm{ev}}}
\newcommand{\od}{^{\mathrm{odd}}}

\def\fhh{\mathfrak{h}}
\def\fgg{\mathfrak{g}}

\def\bC{\mathbb{C}}
\def\bZ{\mathbb{Z}}

\usepackage[centertableaux]{ytableau}
\colorlet{srcol}{black!15}

\makeatletter
\pgfkeys{/ytableau/options,
  noframe/.default = false,
  noframe/.is choice,
  noframe/true/.code = {%
    \global\let\vrule@YT=\vrule@none@YT
    \global\let\hrule@YT=\hrule@none@YT
  },
  noframe/false/.code = {%
    \global\let\vrule@YT=\vrule@normal@YT
    \global\let\hrule@YT=\hrule@normal@YT
  },
  noframe/on/.style = {noframe/true},
  noframe/off/.style = {noframe/false},
}
\makeatother
\ytableausetup{noframe=off,boxsize=1.7em}
\let\ytb=\ytableaushort

\newcommand{\tytb}[1]{{\tiny\ytb{#1}}}

\makeatletter

\newcommand{\trivial}[2][]{\if\relax\detokenize{#1}\relax{\color{red}\vspace{0em} $[$  #2 $]$}\else
\ifx#1h
\ifcsname showtrivial\endcsname{\color{orange} \vspace{0em}  $[$ #2 $]$}\fi\else{\colorWrong argument!}\fi\fi\ignorespaces}

\def\bfdd{{\mathbf{d}}}

\def\fhh{\mathfrak{h}}
\def\fgg{\mathfrak{g}}

\def\bfdd{\mathbf{d}}
\def\bfff{\mathbf{f}}

\def\bfpp{\mathbf{p}}

\def\csqcup{\stackrel{c}{\sqcup}}

\def\cupcol{{{\stackrel{\smash{c}}{\sqcup}\,}}}

\begin{document}
	
 \title[Annihilator variety]{A new proof for the partition algorithm of the  annihilator varieties of  highest weight modules}
 
 \author{Zhanqiang Bai}
\address[Bai]{School of Mathematical Sciences, Soochow University, Suzhou 215006, China}
\email{zqbai@suda.edu.cn}

\author{Jing Jiang}
\address[Jiang]{
School of Mathematical Sciences, East China Normal University, Shanghai 200241, China} 
\email{jjsd6514@163.com}

\author {Yongzhi Luan}
\address[Luan]{School of Mathematics, Shandong University, Jinan 250100, P. R. China; Shenzhen Research Institute of Shandong University, Shenzhen 518057, P. R. China}
\email{luanyongzhi@email.sdu.edu.cn}



\begin{abstract}
	
Let $L(\lambda)$ be a simple highest weight module of a classical Lie algebra  $\mathfrak{g}$ with  highest weight $\lambda-\rho$, where $\rho$ is half the sum of positive roots.
Joseph proved that the associated variety of the annihilator ideal of $L(\lambda)$ (also called the annihilator variety) is the Zariski closure of a nilpotent orbit in $\mathfrak{g}^*$. Recently, Bai--Ma--Wang introduced a partition algorithm to describe this corresponding nilpotent orbit for a given highest weight module $L(\lambda)$.
In this paper, we present a new direct proof of Bai--Ma--Wang's partition algorithm using Sommers duality.
		
	\end{abstract}

 \maketitle

	 \tableofcontents
	\section{Introduction}
	
Let $\mathfrak{g}$ be a simple complex Lie algebra and $\mathfrak{h}$ be a Cartan subalgebra. One of the classical problems about highest weight modules of $\mathfrak{g}$ is to compute their
associated varieties of annihilator ideals. For a highest weight 
 $U(\mathfrak{g})$-module $M$, let $I(M) = \Ann(M)$ be its annihilator 
ideal in $U(\mathfrak{g})$, and let $J(M)$ be the corresponding graded ideal in 
$S(\mathfrak{g}) = \text{gr}(U(\mathfrak{g}))$. The zero set of $J(M)$ in the 
dual vector space $\mathfrak{g}^*$ of $\mathfrak{g}$ is known as \textit{the 
annihilator variety} of $M$, also referred to as \textit{the associated 
variety} of $I(M)$. This is usually denoted by $V(\Ann(M))$.
Borho--Brylinski \cite{BoB1} demonstrated that the annihilator variety of a 
highest weight module with an integral infinitesimal character is irreducible, i.e., it is equal to the Zariski closure of a nilpotent orbit in $\mathfrak{g}^*$ via the Springer correspondence.
Joseph \cite{Jo85} extended this result to the general infinitesimal character case by utilizing the truncated induction ($j$-induction) for representations of Weyl groups.

The study of  annihilator varieties of highest weight module is closely related to many research fields.
Duflo \cite{Du} proved that a primitive ideal (the annihilator ideal of an irreducible representation of $\mathfrak{g}$) of  $U(\mathfrak{g})$ can be realized as the annihilator ideal of a highest weight $U(\mathfrak{g})$-module. Barbasch--Vogan \cite{BarV82} established a one-to-one correspondence among the two-sided cells of the Weyl group $W$, special representations of $W$, and the special nilpotent orbits of $\mathfrak{g}^{*}$; the closures of the corresponding special nilpotent orbits occur as annihilator varieties of highest weight modules with trivial infinitesimal character. 
Joseph \cite{Jo78} found that the dimension of some irreducible module of a finite $W$-algebra is equal to the Goldie rank polynomial of a highest weight $U(\mathfrak{g})$-module; these polynomials were used earlier by Joseph \cite{Jo85} to give a basis for the special representations of an integral Weyl group. 
Further examples and applications can be found in \cite{BoB1,BMSZ-ABCD,GS,GSK,GT,LY}.

The computation of this invariant is generally very complicated. 
For the trivial infinitesimal character this has been done in type $A$ by Joseph \cite{Jo-com}, in types $B$ and $C$ by Garfinkle \cite{Garfinkle3}, and in type $D$ by McGovern and Pietraho \cite{MP23}. Furthermore,  Bai--Ma--Wang \cite{BMW} gave a simple description of the nilpotent orbit associated with the annihilator variety of a highest weight module for classical Lie algebras by using the Robinson--Schensted insertion algorithm utilized in \cite{BaiX,BXX}.  For exceptional cases, Bai--Gao--Wang--Xie  \cite{BGWX} found that the annihilator variety of a highest weight module can be described by using Sommers duality \cite{Sommers}.

In this paper, we present a new proof of the partition algorithm found by Bai--Ma--Wang \cite{BMW} by using Sommers duality.

For  a totally ordered set $ \Gamma $, we  denote by $ \mathrm{Seq}_n (\Gamma)$ the set of sequences $ x=(x_1,x_2,\dots, x_n) $   of length $ n $ with entries $ x_i\in\Gamma $. In this paper, we usually take $\Gamma$ to be $\mathbb{Z}$ or a coset of $\mathbb{Z}$ in $\mathbb{C}$.
Recall that by applying the RS algorithm  to $x\in \mathrm{Seq}_n (\Gamma)$ as in \cite{BaiX}, we obtain a Young tableau $P(x)$.
Let $p(x)=[p_1,...,p_k]$ be the shape of the Young tableau $P(x)$, where $p_i$ is the number of boxes in the $i$-th row of the  Young tableau $P(x)$. Then $p(x)$ is a partition of $n$.  Recall that a partition ${\bf p}=[p_1,...,p_k]$ of $n$ corresponds to a Young diagram, which is an array of $n$ boxes having left-justified rows with the $i$-th row containing $p_i$ boxes for $1 \leq  i \leq k$.

    Now  we let $\fgg$ be a complex simple Lie algebra of type $B_n$, $C_n$, or $D_n$. We identify the dual $\fhh^*$ of the fixed Cartan subalgebra $\fhh$ with $\bC^n$ as usual. 
When $\lambda\in\hs$,  we write $\lambda=(\lambda_1,\dots,\lambda_n)=\sum\limits_{i=1}^n\lam_i\ep_i$, where $\lam_i\in \mathbb{C}$ and $\{\ep_i \mid 1\leq i\leq n\}$ is the canonical basis of the Euclidean space $\mathbb{R}^n$.  

    For $\lambda\in \fhh^*$, let	$[\lambda] $ be the collection of  maximal subsequences $x=(\lambda_{i_1},\lambda_{i_2},\dots, \lambda_{i_k})$ of  $ \lambda$ such that 
    \begin{itemize}
    \item $i_1 < i_2< \cdots < i_k$ and
    \item $\lambda_{i_l}-\lambda_{i_{l+1}} \in \bZ$ or $\lambda_{i_l} + \lambda_{i_{l+1}} \in \bZ$ for all $1\leq l<k$. 
    \end{itemize}
     We set $ [\lambda]_1 $ (resp. $ [\lambda]_2 $) to be the subset of $ [\lambda] $ consisting of sequences with  all entries belonging to $ \mathbb{Z} $ (resp. $ \frac12+\mathbb{Z} $).
    Note that there is at most one element in $[\lambda]_1 $ (resp. $[\lambda]_2 $); we denote it by  $(\lambda)_{(0)}$ (resp. $(\lambda)_{(\frac{1}{2})}$) if it exists. 
     
	We set 
 \[
 [\lambda]_{1,2}=[\lambda]_1\cup [\lambda]_2, \quad [\lambda]_3=[\lambda]\setminus[\lambda]_{1,2}. 	
 \]

	For $ x=(\lambda_{i_1}, \lambda_{i_2},\dots, \lambda_{i_r})\in[\lambda] $, we define
	$${x}^-=(\lambda_{i_1},\lambda_{i_2},\dots,\lambda_{i_{r-1}}, \lambda_{i_{r}},-\lambda_{i_{r}},-\lambda_{i_{r-1}},\dots,-\lambda_{i_{2}},-\lambda_{i_{1}}).$$


For $ x=(\lambda_{i_1}, \lambda_{i_2},\dots, \lambda_{i_r})\in[\lambda]_3 $, let $y=(\lambda_{j_1}, \lambda_{j_2},\dots, \lambda_{j_p}) $ be the maximal subsequence of $ x $ such that $ j_1=i_1 $ and the difference of any two entries of $ y$ is an integer, and let  $z= (\lambda_{k_1}, \lambda_{k_2},\dots, \lambda_{k_q}) $ be the subsequence obtained by deleting $ y$ from $ x $, which is possibly empty.
	Define
	$$  \tilde{x}=(\lambda_{j_1}, \lambda_{j_2},\dots, \lambda_{j_p}, -\lambda_{k_q}, -\lambda_{k_{q-1}},\dots ,-\lambda_{k_1}).  $$
	
	Note that for   $ x=(\lambda_{i_1}, \lambda_{i_2},\dots, \lambda_{i_r})\in[\lambda]_{1,2} $, we obtain a Young tableau $P(x^-)$ by using the RS  algorithm.	This tableau yields a partition  $p(x^-)$, which corresponds to a special  partition for a special nilpotent orbit via the Springer correspondence.

Similarly, for $(\lambda)_{(\frac{1}{2})}$, we obtain a partition $p((\lambda)_{(\frac{1}{2})}^-)$, which gives a special partition of type $D$. This special type $D$ partition yields a $C$-type metaplectic special partition by computing the dual of $D$-collapse of its dual partition. See \cite{BMSZ-typeC}.  The definition of the collapse of partitions will be given in \S \ref{orbits-classical}.

Let $X$  denote the type of a  Lie algebra, i.e., $X=B$, $C$, or $D$. Let $\bfdd\cup \bfff$	denote the partition consisting of all row parts of $\bfdd$ and $\bfff$. Let  $\bfdd\cupcol \bfff$ denote the partition given by $$\bfdd\cupcol\bfff = (\bfdd^t\cup \bfff^t)^t.$$

	\begin{Thm}[{\cite[Thm. 1.6]{BMW}}]\label{pa-suanfa}
		Suppose $\mathfrak{g} = \mathfrak{so}(2n+1, \mathbb{C})$, $\mathfrak{sp}(n, \mathbb{C})$, or $\mathfrak{so}(2n, \mathbb{C})$, $\lambda\in \mathfrak{h}^*$, and
		$[\lambda]=(\lambda)_{(0)} \cup (\lambda)_{(\frac{1}{2})}\cup [\lambda]_3$
		with $[\lambda]_3=\{{\lambda}_{{Y}_1},\dots,{\lambda}_{{Y}_m}\}$. Let
		\begin{enumerate}
			\item $\mathbf{p}_{0}$ be the $X$-type special partition associated with
			$(\lambda)_{(0)}$;
			\item ${\mathbf p}_{\frac{1}{2}}$ be the $C$-type special  partition (for type $B$) or $C$-type metaplectic special partition (for types $C$ and $D$) associated with
			$(\lambda)_{(\frac{1}{2})}$;
			\item $\mathbf{p}_{i}$ be the $A$-type partition associated with
			$\tilde{\lambda}_{Y_i}$.
		\end{enumerate}
		Let $\mathbf{p}_{\lambda}$ be the $X$-collapse of
 \[
\mathbf{p}_0  \overset{c}{\sqcup}  \mathbf{p}_{\frac{1}{2}}  \overset{c}{\sqcup} \left( \overset{c}{\sqcup_i}  2\mathbf{p}_i \right)
\]
		Then we have
		\[
		V(\Ann (L(\lambda)))=\overline{\mathcal{O}}_{\mathbf{p}_{\lambda}}.
		\]
 When $n$ is even and ${\mathcal{O}}_{\mathbf{p}_{\lambda}}$ is a very even orbit,
 its numeral is determined by the number $k(\lambda)$ of very even orbits with numeral II in the set of very even orbits of type $D$:    $\{\mathcal{O}_{\bfpp_{0}}, \mathcal{O}_ {{\bf p}_{\frac{1}{2}}},  \mathcal{O}_{2\bfpp_i}| 1\leq i\leq m  \}$. 
So $\mathcal{O}_{{\bf p}_{\lambda}}$ will be of type I if $k(\lambda)\equiv 0 
  ~(\mathrm{mod}~ {2})$ and of type II if $k(\lambda)\equiv 1 
  ~(\mathrm{mod}~ {2})$.
 \end{Thm}

Originally, this partition algorithm was discovered via computer experiments, and a proof using $j$-inductions was given in \cite{BMW}.
In this paper, we provide a new proof of this partition algorithm by using Sommers duality.


	This paper is organized as follows. In \S \ref{subsec:gkh}, we give some necessary preliminaries on annihilator varieties of highest weight modules, Sommers duality, and nilpotent orbits. In \S \ref{som--duality}, we recall the computation of annihilator varieties of highest weight modules by using the Sommers duality for classical groups. In \S \ref{new-proof}, we give the new proof of the partition algorithm by using the Sommers duality.


\section{Notation and Preliminary Results}\label{subsec:gkh}
Let $\mathfrak{g}$ be a simple complex Lie algebra and $\mathfrak{h}$ be a Cartan subalgebra.
Let $\Phi^+\subset\Phi$ be the set of positive roots determined by a Borel subalgebra $\mathfrak{b}$ of $\mathfrak{g}$. This gives us the Cartan decomposition $\mathfrak{g}=\mathfrak{n}\oplus\mathfrak{h}\oplus\mathfrak{n}^-$. Denote by $\Delta$ the set of simple roots in $\Phi^+$.
Any subset $I\subset \Delta$ generates a subsystem $\Phi_I\subset\Phi$.
Let $\mathfrak{p}_I$ be the standard parabolic subalgebra corresponding to $I$ with Levi decomposition $\mathfrak{p}_I=\mathfrak{ l}_I\oplus\mathfrak{u}_I$. We frequently drop the subscript $I$ if there is no confusion.
	
Let $F(\lambda)$ be a finite-dimensional irreducible $\mathfrak{l}$-module with highest weight $\lambda -\rho\in\mathfrak{h}^*$, where $\rho$ is half the sum of positive roots in $\Phi^+$. It can also be viewed as a	$\mathfrak{p}$-module with a trivial $\mathfrak{u}$-action. The {\it generalized Verma module} $N_I(\lambda)$ is defined by
\[
	N_I(\lambda):=U(\mathfrak{g})\otimes_{U(\mathfrak{p})}F(\lambda).
\]
The simple quotient of $N_I(\lambda)$ is denoted by $L(\lambda)$, which is a highest weight module with highest weight $\lambda-\rho$.

\subsection{Integral Weyl group}\label{integral2}
	Let $\langle-, -\rangle: \mathfrak{h} \times \mathfrak{h}^* \to \mathbb{C}$ be the canonical pairing. For $\lambda\in\mathfrak{h}^*$, define 
\begin{equation*}
\Phi_{[\lambda]}:=\{\alpha\in\Phi\mid\langle\lambda, \alpha^\vee\rangle\in\mathbb{Z}\},
\end{equation*}
where $ \alpha^\vee$ is the coroot associated with the root $\alpha \in \Phi$. Set
	\[
	W_{[\lambda]}:=\{w\in W\mid w\lambda-\lambda\in \mathbb{Z}\Phi\}.
	\]
	Then $\Phi_{[\lambda]}$ is a root system with Weyl group $W_{[\lambda]}$
	(e.g., \cite{Hum08}).
	Let $\Delta_{[\lambda]}$ be the simple system of $\Phi_{[\lambda]}$.  Set $J=\{\alpha\in\Delta_{[\lambda]}\mid\langle\lambda, \alpha^\vee\rangle=0\}$. Denote by $W_J$ the Weyl group generated by the reflections $s_\alpha$ with $\alpha\in J$. Let $\ell_{[\lambda]}$ be the length function on $W_{[\lambda]}$. Thus $\ell_{[\lambda]}=\ell$ when $\lambda$ is integral. Put
	\begin{equation*}\label{ceq1}
	W_{[\lambda]}^J:=\{w\in W_{[\lambda]}\mid \ell_{[\lambda]}(ws_\alpha)=\ell_{[\lambda]}(w)+1\ \mbox{for all}\ \alpha\in J\}.
	\end{equation*}
	Thus $W_{[\lambda]}^J$ consists of the shortest representatives of the cosets $wW_J$ with $ w\in W_{[\lambda]} $. When $\lambda$ is integral, we simply write $W^J:=W_{[\lambda]}^J$ .
	
	A weight $ \mu\in\hs $ is called \textit{anti-dominant} if $ \bil{\mu}{\al} \notin\mathbb{Z}_{>0}$ for all $ \al\in\Phi^+ $. For any $\lambda\in\mathfrak{h}^*$, there exists a unique anti-dominant weight $\mu\in\hs$ and a unique $w\in W_{[\mu]}^J$ such that $\lambda=w\mu$.
\begin{Prop}[{\cite[Prop. 3.5]{Hum08}}]\label{anti}
    Let $\lambda\in \mathfrak{h}^*$, with corresponding root system $\Phi_{[\lambda]}$ and Weyl group $W_{[\lambda]}$. Let $\Delta_{[\lambda]}$ be the simple system of $\Phi_{[\lambda]} \cap \Phi^+$  in $\Phi_{[\lambda]}$.  Then $\lambda$ is antidominant if and only if one of the
following three equivalent conditions holds:
\begin{enumerate}
    \item $\langle\lambda, \alpha^\vee\rangle\leq 0$ for all $\alpha \in \Delta_{[\lambda]}$.
    \item $\lambda\leq s_{\alpha}\lambda$ for all $\alpha \in \Delta_{[\lambda]}$.
    \item  $\lambda\leq w\lambda$ for all $w \in W_{[\lambda]}$.
\end{enumerate}
Therefore there is a unique antidominant weight in the orbit $W_{[\lambda]}\lambda$.
\end{Prop}

 \subsection{Annihilator varieties of highest weight modules}
	
	Let $M$ be a finitely generated $U(\mathfrak{g})$-module. Fix a finite-dimensional generating subspace $M_0$ of $M$. Let $U_{n}(\mathfrak{g})$ be the standard filtration of $U(\mathfrak{g})$. Set $M_n=U_n(\mathfrak{g})\cdot M_0$ for $n \geq 0$ and put $M_{-1}=0$. Then
	\(
	\text{gr} (M)=\bigoplus\limits_{n=0}^{\infty} \text{gr}_n M,
	\)
	where $\text{gr}_n M=M_n/{M_{n-1}}$. Thus $\text{gr}(M)$ is a graded module over $\text{gr}(U(\mathfrak{g}))\simeq S(\mathfrak{g})$.
	We use $\Ann (M)$ to denote the two-sided ideal of $U(\mathfrak{g})$ consisting of elements that act by zero on $M$.

	\begin{Defff}
		The  \textit{associated variety} of $M$ is defined by
		\begin{equation*}
		V(M):=\{X\in \mathfrak{g}^* \mid f(X)=0 \text{ for all~} f\in \operatorname{Ann}_{S(\mathfrak{g})}(\operatorname{gr} M)\}.
		\end{equation*}
	\end{Defff}
	
	The above  definition is independent of the choice of $M_0$ (e.g., \cite{NOT}).
	
	\begin{Defff} Let $\mathfrak{g}$ be a finite-dimensional semisimple Lie algebra. Let $I$ be a two-sided ideal in $U(\mathfrak{g})$. Then $\text{gr}(U(\mathfrak{g})/I)\simeq S(\mathfrak{g})/\text{gr}I$ is a graded $S(\mathfrak{g})$-module. Its annihilator is $\text{gr}I$. We define its associated variety by
		$$V(I):=V(U(\mathfrak{g})/I)=\{X\in \mathfrak{g}^* \mid p(X)=0\ \mbox{for all $p\in {\text{gr}}I$}\}.
		$$
	\end{Defff}
	
	Following \cite{GSK}, $V(\Ann (M))$ is called the \textit{annihilator variety} of the $U(\mathfrak{g})$-module $M$.

	Recall that the Weyl group $ W  $ is a Coxeter group generated by $ S=\{s_\al\mid\al\in\Delta \} $. Given an indeterminate $q$, the Hecke algebra $ \mc{H} $ over $ \mathcal{A} :=\mathbb{Z}[q,q^{-1}]$ is generated by $ T_w $, $ w\in W $ with relations \[
T_wT_{w'}=T_{ww'} \text{ if }\ell(ww')=\ell(w)+\ell(w'),
\]
\[
\text{and }(T_s+q^{-1})(T_s-q)=0 \text{ for any }s\in S.
\]
For each $w \in W$, the unique element $ C_w $ such that
\[
\overline{C_w}=C_w,\qquad C_w\equiv T_w \mod{\mc{H}_{<0}}
\]
forms, as $w$ varies, the \textit{Kazhdan--Lusztig} (KL) \textit{basis} of $ \mc{H} $, where $ \bar{\,} :\mc{H}\rar\mc{H}$ is the bar involution such that $ \bar{q}=q^{-1} $, $ \overline{T_w} =T_{w^{-1}}^{-1}$, and $ \mc{H}_{<0}=\bigoplus_{w\in W}\mathcal{A}_{<0}T_w $ with $ \mathcal{A}_{<0}=q^{-1}\mathbb{Z}[q^{-1}] $.

If $ C_y $ occurs in the expansion of $ hC_w $ (resp. $C_wh$) with respect to the KL basis for some $ h\in\mc{H} $, then we write $ y\leftarrow_L w $ (resp. $ y\leftarrow_R w $). Extend $ \leftarrow_L $ (resp. $ \leftarrow_R $) to a preorder $ \leq_L $ (resp. $\leq _R$) on $ W $. For $x, w\in W$, write $x \leq_{LR} w$ if there exists a sequence $x=w_1, \cdots, w_n=w$ such that, for every $1\leq i<n$, either $w_i\leq_L w_{i+1}$ or $w_i\leq_R w_{i+1}$. Let $\stackrel{L}{\sim}$, $\stackrel{R}{\sim}$, $\stackrel{LR}{\sim}$ be the equivalence relations associated with $\leq_L$, $\leq_R$, $\leq_{LR}$ (for example, $x\stackrel{L}{\sim}w$ if and only if $x\leq_L w$ and $w\leq_Lx$). The equivalence classes of $W$ under $\stackrel{L}{\sim}$, $\stackrel{R}{\sim}$, $\stackrel{LR}{\sim}$ are called \textit{left cells}, \textit{right cells} and \textit{two-sided cells} respectively. More details about cells can be found in \cite{KL}.

Let $G$ be a connected simple algebraic group over $\mathbb{C}$ with Lie algebra $\mathfrak{g}$. Denote by $\mathcal{N}(G)$ or $\mathcal{N}(\mathfrak{g})$ the set of nilpotent orbits in $\mathfrak{g}$. For $\mathcal{O} \in \mathcal{N}(G)$, we write $A_\mathcal{O}:=Z_{G}(x)/(Z_{G}(x))^o$ for the connected component group of the centralizer subgroup  of an element $x\in \mathcal{O}$. Consider
$$\mathcal{N}^{\rm en}(G):=\{(\mathcal{O}, \eta): \mathcal{O} \in \mathcal{N}(G), \eta \in {\rm Irr}(A_\mathcal{O}) \}.$$
The Springer correspondence gives an injective map
\begin{equation} \label{Spr}
  {\rm Spr}_G: {\rm Irr}(W(G)) \longrightarrow \mathcal{N}^{\rm en}(G)  
\end{equation}
denoted by ${\rm Spr}_G(\sigma) = (\mathcal{O}_{\rm Spr}^G(\sigma), \eta_\sigma)$. 
For every orbit $\mathcal{O} \in \mathcal{N}(G)$, the pair $(\mathcal{O}, \mathbf{1})$ always lies in the image of ${\rm Spr}_G$, and thus we have a well-defined map
$${\rm Spr}_{G, \mathbf{1}}^{-1}: \mathcal{N}(G) \longrightarrow {\rm Irr}(W(G))$$
given by ${\rm Spr}_{G,\mathbf{1}}^{-1}(\mathcal{O}):= {\rm Spr}_G^{-1}(\mathcal{O}, \mathbf{1})$.  If $\sigma \in {\rm Irr}(W(G))$ is a special representation, then one has ${\rm Spr}_G(\sigma) = (\mathcal{O}_{\rm Spr}^G(\sigma), \mathbf{1})$, and we call $\mathcal{O}_{\rm Spr}^G(\sigma)$ a special nilpotent orbit.

In fact, it can be seen that ${\rm Spr}_G$ actually depends only on the Lie algebra $\mathfrak{g}$, and thus we may write all the above using $\mathfrak{g}$ as
$${\rm Spr}_\mathfrak{g}:={\rm Spr}_G, \quad \mathcal{O}_{\rm Spr}^\mathfrak{g}:=\mathcal{O}_{\rm Spr}^G, \quad {\rm Spr}_{\mathfrak{g}, \mathbf{1}}^{-1}:={\rm Spr}_{G, \mathbf{1}}^{-1}.$$

Let $W$ be the Weyl group of $\mathfrak{g}$. 
We denote by $L_w$ the simple highest weight $\mathfrak{g}$-module with highest weight $-w\rho-\rho$, where $w\in W$. 
Set $I_w=\Ann(L_w)$. 
Then, according to Joseph~\cite{Jo79} and Vogan~\cite{Vo80}, $I_w=I_x$ if and only if $w\stackrel{L}{\sim}x$. 
Borho and Brylinski~\cite{BoB1} proved that the annihilator variety of $L_w$ is irreducible; i.e., it is the closure of a single nilpotent orbit. 
We write $V(I_w)=V(\Ann(L_w))=\overline{\mathcal{O}}_w$.

From~\cite{BoB1} and~\cite{Ta}, there exists a bijection between the two-sided cells of the Weyl group $W$ and the special nilpotent orbits. 
In other words, $\mathcal{O}_w=\mathcal{O}_y$ if and only if $w\stackrel{LR}{\sim}y$. 

Let $\lambda\in\mathfrak{h}^*$ with associated root system $\Phi_{[\lambda]}$ and Weyl subgroup $W_{[\lambda]}$. 
Then $\lambda$ is an integral weight for $\Phi_{[\lambda]}$. 
Suppose that $\lambda=w_{\lambda}\mu$, where $\mu$ is antidominant and $w_{\lambda}\in W_{[\mu]}^J$. 
We denote by $\pi_{w_{\lambda}}$ the special representation of $W_{[\lambda]}$ corresponding to the two-sided cell of $W_{[\lambda]}$ containing $w_{\lambda}$.

Using the Springer correspondence and the $j$-induction operator, Joseph extended the result of~\cite{BoB1} to an arbitrary infinitesimal character, which gives the following result.

\begin{Prop}[{\cite[Thm.~3.10]{Jo85}}]\label{2dim}
Let $\mu\in\mathfrak{h}^*$ be an antidominant element and $w\in W_{[\mu]}^J$. 
Then $\tilde{\pi}_w:=j_{W_{[\mu]}}^W(\pi_w)$ is an irreducible $W$-module. 
Via the Springer correspondence, this $W$-module $\tilde{\pi}_w$ corresponds to a nilpotent orbit $\mathcal{O}_{\tilde{\pi}_w}$ with trivial local system. Furthermore,
$$V(\Ann(L(w\mu)))=\overline{\mathcal{O}_{\mathrm{Spr}}^{\mathfrak{g}}(\tilde{\pi}_w)}.$$
\end{Prop}

Henceforth, for a highest weight module $L(\lambda)$, Proposition~\ref{2dim} yields 
$$V(\Ann(L(\lambda)))=\overline{\mathcal{O}_{\Ann(L(\lambda))}}$$
for a nilpotent orbit $\mathcal{O}_{\Ann(L(\lambda))}$ that is uniquely determined by $L(\lambda)$. 
We call $\mathcal{O}_{\Ann(L(\lambda))}$ the associated nilpotent orbit, or annihilator orbit, of $L(\lambda)$.

\subsection{Robinson--Schensted insertion algorithm}\label{R-S} 
A generalized version of the  Robinson--Schensted insertion algorithm is presented in \cite{BaiX}, which we will utilize in our paper.
This RS algorithm or correspondence was  originally discovered by Robinson \cite{Rob38} and then independently rediscovered in a different form by Schensted \cite{Sc61}. 
 
Recall that we  denote by $ \mathrm{Seq}_n (\Gamma)$ the set of sequences $ x=(x_1,x_2,\dots, x_n) $   of length $ n $ with $ x_i\in\Gamma $, where $\Gamma$ is a totally ordered set. We obtain a Young tableau $P(x)$ by applying the following Robinson--Schensted insertion algorithm  to $x\in \mathrm{Seq}_n (\Gamma)$. 
 \begin{Defff}[Robinson--Schensted insertion algorithm]
For an element  $ x \in  \mathrm{Seq}_n (\Gamma)$, we write  $x=(x_1,\dots,x_n)$. We associate to $x $ a  Young tableau  $ P(x) $ as follows. Let $ P_0 $ be an empty Young tableau. Assume that we have constructed the Young tableau $ P_k $ associated to $ (x_1,\dots,x_k) $, $ 0\leq k<n $. Then $ P_{k+1} $ is obtained by adding $ x_{k+1} $ to $ P_k $ as follows. First, we add $ x_{k+1} $ to the first row of $ P_k $ by replacing the leftmost entry $ x_j $ in the first row which is \textit{strictly} bigger than $ x_{k+1} $.  (If there is no such entry $ x_j $, we just add a box with entry $x_{k+1}  $ to the right side of the first row, and end this process). Then add this $ x_j $ to the next row in the same way that we  added $x_{k+1} $ to the first row.  Finally, we put $P(x)=P_n$.

\end{Defff}

We use $p(x)=(p_1,\dots, p_k)$ to denote the shape of $P(x)$, where $p_i$ is the number of boxes in the $i$-th row of  $P(x)$.
When $|p(x)|=\sum\limits_{1\leq i\leq k} p_i=N$, $p(x)$ is a partition of $N$ and we still denote this partition by $p(x)=[p_1,\dots, p_k]$.

In general, the Robinson--Schensted insertion algorithm is abbreviated to the RS algorithm. 

\subsection{Nilpotent orbits of classical types}\label{orbits-classical}
Let $\fgg$ be a simple complex Lie algebra.
Recall that nilpotent orbits in $\fgg$ are parameterized by partitions (with additional labels in type $D$).

    The partition types of nilpotent orbits of classical types  are given in the following propositions.

\begin{Prop}[{\cite[Thm. 5.1.1 \& Thm. 6.3.2]{CM}}]\label{Orbit-A}
		Nilpotent orbits in $\mathfrak{sl}{(n,\mathbb{C})}$ are in one-to-one correspondence with the set of partitions of $n$. Every nilpotent orbit (corresponding to a partition ${\bf p}$) of type $A$ is special.
	\end{Prop}

    \begin{Prop}[{\cite[Thm. 5.1.2 \& Prop. 6.3.7]{CM}}]\label{Orbit-B}
		Nilpotent orbits in $\mathfrak{so}{(2n+1,\mathbb{C})}$ are in one-to-one correspondence with the set of partitions of $2n+1$ in which even parts occur with even multiplicity. A nilpotent orbit (corresponding to a partition ${\bf p}$) of type $B$ is special if and only if  ${\bf p}^t$ corresponds to a nilpotent orbit of type $B$.
	\end{Prop}
	
	\begin{Prop}[{\cite[Thm. 5.1.3 \& Prop. 6.3.7]{CM}}]\label{Orbit-C}
		Nilpotent orbits in $\mathfrak{sp}{(n,\mathbb{C})}$ are in one-to-one correspondence with the set of partitions of $2n$ in which odd parts occur with even multiplicity. A nilpotent orbit (corresponding to a partition ${\bf p}$) of type $C$ is special if and only if  ${\bf p}^t$ corresponds to a nilpotent orbit of type $C$.
	\end{Prop}
    
    \begin{Prop}[{\cite[Thm. 5.1.4 \& Prop. 6.3.7]{CM}}]\label{orbit-D}
		Nilpotent orbits in $\mathfrak{so}{(2n,\mathbb{C})}$ are in one-to-one correspondence with the set of partitions of $2n$ in which even parts occur with even multiplicity,
		except that each ``very even" partition ${\bf p}$ (consisting of only even parts) corresponds to two orbits, denoted by $\mathcal{O}^I_{\bf p}$ and $\mathcal{O}^{II}_{\bf p}$. A nilpotent orbit (corresponding to a partition ${\bf p}$) of type $D$ is special if and only if  ${\bf p}^t$ corresponds to a nilpotent orbit of type $C$.
	\end{Prop}
	
 \begin{Prop}[{\cite[Thm. 8.2.4 ]{CM}}]\label{Orbit-dis}
	The distinguished orbits of $\mathfrak{g}$ are as follows:
\begin{enumerate}
    \item If $\mathfrak{g}$ is of type $A_{n-1}$, then the only distinguished orbit is the regular orbit with the partition $\mathbf{p}=[1^n]$.
    \item If $\mathfrak{g}$ is of type $B$, $C$, or $D$, then an orbit is distinguished if and only if its partition has no repeated parts. Thus, the partition of a distinguished orbit in types $B$ and $D$ consists only of odd parts, each with multiplicity one, while the partition of a distinguished orbit in type $C$ consists only of even parts, each with multiplicity one.
\end{enumerate}
\end{Prop}  

Given two partitions ${\bf d}=[d_1,\dots,d_s]$ and ${\bf f}=[f_1,\dots,f_t]$ of $N$, we say that ${\bf d}$ {\it dominates } ${\bf f}$ and write ${\bf d}\geq {\bf f}$ if the following condition holds:
\begin{equation}\label{domi}
    \sum\limits_{1\leq j\leq l}d_j\geq \sum\limits_{1\leq j\leq l}f_j
\end{equation}
for all $l\geq 1$, where missing parts are understood to be zero. Equivalently, it suffices to check $1\leq l\leq \max\{s,t\}$.

	\begin{Defff}[Collapse]	
		Let ${\bf d}=[d_1,\dots,d_k]$ be a partition of $2n$. There is a unique maximal partition of $2n$ of type $D_n$ dominated by ${\bf d}$ with respect to the dominace order. If  ${\bf d}$ is not a partition of type $D_n$, then one of its even parts must occur with an odd multiplicity. Let $q$ be the largest such part. Then replace the last occurrence of $q$ in ${\bf d}$ by $q-1$ and the first subsequent part $r$ strictly less than $q-1$ by $r+1$. Repeat this process until a partition of type $D_n$ is obtained. This new partition of type $D_n$
		is called the {\it $D$-collapse} of 	${\bf d}$, and we denote it by ${\bf d}_D$. Similarly, one can define the $B$-collapse and $C$-collapse of ${\bf d}$.
	\end{Defff}	

Further properties of the  collapse  of partitions can be found in \cite{CM}.

When no confusion can arise, we identify partitions and Young diagrams with nilpotent orbits. For a Young diagram $P$ with shape $p$, use $(k, l)$ to denote the box in the $k$-th row and
the $l$-th column. We say that the box $(k, l)$ is even (resp. odd) if $k + l$ is even (resp. odd).

\begin{Defff}
		By removing all the odd boxes from a Young  diagram  $P$, we obtain a  diagram $ P^{\ev} $  consisting of only even boxes and inheriting the filling from $ P $. We say $ P^{\ev} $ is the \textit{even  diagram} of $ P $. Similarly, we define $ P^{\od} $. These  diagrams are called {\it Hollow  diagrams} of $P$ in \cite{BXX}.
	\end{Defff}
	
\begin{ex}\label{evenbox}
		Let $P$ be a Young diagram with shape $ p=[3,3,2,2,1] $. The even and odd boxes in $P$ are marked as follows:
		\[
		\tytb{EOE,OEO,EO,OE,E}.
		\]
		Then we have
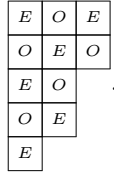
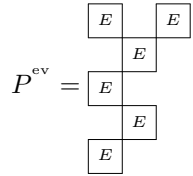
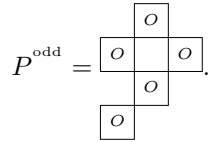
		\[
		P^{\ev}=\tytb{E\none E, \none E, E, \none E, E}
		\]	
and
\[
		P^{\od}=\tytb{\none O\none,O\none O,\none O, O}.
		\]	
	\end{ex}
	
	By the result in \cite{BXX} or \cite{BMW},  from an integral weight $\lambda$ (or a classical Weyl group element $w_{\lambda}\in W$), we obtain a special partition ${\bf p}^s$ of classical type, which corresponds to the special nilpotent orbit $\mathcal{O}_{w_{\lambda}}$. This algorithm is equivalent to the H-algorithm defined in \cite[\S 8.2]{BMW}. Roughly speaking, the special partition ${\bf p}^s$  has the same odd (resp. even) boxes for types $B$
and $C$ (resp. type $D$) with the partition ${\bf p}=p(\lambda^-)$. We recall the definition of the H-algorithm of type $B$ in the following.

 \begin{Defff}[H-algorithm of type $B$]
    Let $\lambda\in \mathfrak{h}^*$ be an integral weight of type $B_n$ and ${\bf p}=p(\lambda^-)$ be a partition (whose Young diagram is $P$) of $2n$, then we can get a special partition  ${\bf p}^s$ of type $B_n$ by the following steps:
    \begin{enumerate}
\item Construct  the Hollow diagram $P^{\od}$ consisting of odd boxes;
\item Label the rows starting from $1$ but skip any consecutive rows ending with the shape  $\tytb{O,\none O}$;
\item Keep even labeled rows unchanged and put $\tytb{E}$ at the end of each odd labeled row;
\item Fill the holes. If there are only $2n$ boxes in our new Young diagram, we put a box $\tytb{E}$  below the last row and we are done. If there are  $2n+1$ boxes in our new Young diagram, we are done.
\end{enumerate}
We call the above algorithm the  {\it H-algorithm of type $B$}.
\end{Defff}

Similarly, there are H-algorithms of type $C$ and type $D$.  See \cite[\S 8.2]{BMW} for more details.

\subsection{Sommers duality algorithm for annihilator varieties of highest weight modules}

Let $\Phi$ be a root system and let
\[
\alpha_0=\sum_{i=1}^nh_i\alpha_i
\]
be the highest root of $\Phi$, where $h_i$ are non-negative integers and $\Delta=\{\alpha_i\mid 1\leq i\leq n\}$ is the set of simple roots in $\Phi^+$. 
For every $1\leq i, j \leq n$, we set
$$\tilde{\Phi}(i):=\Phi \cup \{-\alpha_0\} - \{\alpha_i\}, \quad \Phi(j):=\Phi - \{\alpha_j\}.$$

\begin{Defff}\label{thmbd}
	Let $\Phi$ be irreducible. A root subsystem of $\Phi$ is called pseudo-maximal if it is a proper root subsystem and is equal (up to the action of $W$) to a maximal element of the set
$\{\tilde{\Phi}(i), \Phi(j): \ 1\leq i, j \leq n \}$.
\end{Defff}

Let $G$ be an algebraic group over $\mathbb{C}$ with Lie algebra $\mathfrak{g}$. Recall that a pseudo-Levi subgroup $M \subseteq G$ is the connected  centralizer subgroup of a semisimple element of $G$. Such a group is essentially characterized by the fact that the Dynkin diagram of its root system is a subset of the extended Dynkin diagram of the root system of $G$. We consider pairs of the form $(M, \mathcal{O}_M)$, where $M$ is a pseudo-Levi subgroup and $\mathcal{O}_M \in \mathcal{N}(M)$ is a distinguished orbit. Let $G^\vee$ be the Langlands dual group of $G$. Then Sommers \cite{Sommers} defined a natural map
$$d_{\rm Som}^G: \{(M, \mathcal{O}_M) \} \longrightarrow \mathcal{N}(G^\vee)$$
that extends the Barbasch--Vogan dual map $d_{\rm BV}: \mathcal{N}(G) \to \mathcal{N}^{\rm sp}(G^\vee)$, where $\mathcal{N}^{\rm sp}(G) \subset \mathcal{N}(G)$ denotes the subset of special nilpotent orbits of $G$. We recall this map in the following.

 First, one has the Lusztig--Spaltenstein map $d_{\rm LS}^G: \mathcal{N}(G) \longrightarrow \mathcal{N}^{\rm sp}(G)$. Moreover, one has a natural isomorphism 
$$\tau_{G, G^\vee}: W(G) \longrightarrow W(G^\vee)$$
of the two Weyl groups, which induces a bijection (using the same notation)
$$\tau_{G, G^\vee}: {\rm Irr}(W(G)) \longrightarrow {\rm Irr}(W(G^\vee)).$$
It furthermore induces a bijection (by abuse of notation again)
$$\tau_{G, G^\vee}: \mathcal{N}^{\rm sp}(G) \longrightarrow \mathcal{N}^{\rm sp}(G^\vee).$$
One has $d_{\rm BV}:= \tau_{G, G^\vee} \circ d_{\rm LS}^G$. All these maps mentioned above fit into the following commutative diagram:
$$\begin{tikzcd}
\mathcal{N}^{\rm sp}(G) \ar[r, "{\tau_{G,G^\vee}}"] & \mathcal{N}^{\rm sp}(G^\vee)  \ar[r, hook] & \mathcal{N}(G^\vee) \\
& \mathcal{N}(G) \ar[lu, "{d_{\rm LS}^G}"] \ar[u, "{d_{\rm BV}^G}"] \ar[r, hook] & {\{(M, \mathcal{O}_M) \} } \ar[u, "{d_{\rm Som}^G}"] .
\end{tikzcd}
$$
Here the bottom injection is provided by the Bala--Carter classification of nilpotent orbits, asserting that every $\mathcal{O} \in \mathcal{N}(G)$ corresponds uniquely to a $G$-conjugacy class of pair $(L, \mathcal{O}_L)$, where $L \subseteq G$ is a Levi subgroup and $\mathcal{O}_L$ a distinguished orbit for $L$. In this case, we write
$${\rm sat}_L^G(\mathcal{O}_L):= G \cdot \mathcal{O}_L \in \mathcal{N}(G)$$ for the nilpotent orbit thus obtained. 

Let $M \subset G$ be a pseudo-Levi subgroup and $\mathcal{O}_{M}$ a distinguished orbit for $M$. Then it was shown by Sommers \cite[\S 6]{Sommers} that 
\begin{equation} \label{E:dSom}
d_{\rm Som}^{G}(M, \mathcal{O}_{M}) = \mathcal{O}_{\rm Spr}^{G^\vee} \circ \tau_{G, G^\vee} \circ j_{W(M)}^{W(G)} \circ {\rm Spr}_{M, \mathbf{1}}^{-1} \circ d_{\rm LS}^{M}(\mathcal{O}_{M}),
\end{equation}
where $j_{W(M)}^{W(G)}$ denotes the $j$-induction operator, see \cite[\S 11]{Ca85} or \cite{BMW} for more details about this operator.

We have the following algorithm to compute the annihilator variety of a highest weight module.

\begin{Thm}[{\cite[Thm. 1.2]{BGWX}}]\label{som-alg}
Keep the notation as above.
Let $H^\vee \subseteq G^\vee$ be a pseudo-Levi subgroup with root system $\Phi_{[\lambda]}^\vee$. Let $H$ be its Langlands dual group, whose root system is then $\Phi_{[\lambda]}$.
Now $\pi_{w_\lambda} \in {\rm Irr}(W(H))$ is a special representation of $W(H)$. Let $\pi^{\vee}_{w_\lambda}$ be the corresponding special representation of $W(H^{\vee})$ obtained via the canonical isomorphism $W(H) \simeq W(H^\vee)$. Furthermore, let $L^\vee\subset H^\vee$ be a Levi subgroup of $H^\vee$ and $\mathcal{O}_{L^\vee}$ a distinguished nilpotent orbit of $L^\vee$ such that $d_{\rm LS}^{H^\vee}(\mathcal{O}(\pi^{\vee}_{w_\lambda})) = H^\vee\cdot \mathcal{O}_{L^\vee}$.
Then 
$$\mathcal{O}_{{\rm Ann}(L(\lambda))} = d_{\rm Som}^{G^\vee}(L^{\vee}, \mathcal{O}_{L^\vee}),$$
where $d_{\rm Som}^{G^\vee}$ is the duality map given in \cite{Sommers}.
\end{Thm}

\section{Sommers duality algorithm for classical types}\label{som--duality}

In this section, we  recall Sommers duality in \cite{Sommers} for classical groups of types other than $A$, and then recall the Sommers duality algorithm given in \cite{BGWX} to compute the annihilator varieties of highest weight modules for classical Lie algebras. 

\subsection{Sommers duality for classical groups}
Let $\parti(m)$ denote the set of partitions $\mathbf{p} = [p_1, \dots, p_k]$ of $m$. Then the set of nilpotent orbits of $\mathfrak{g}$ is in bijection with $\parti(n)$ when $\mathfrak{g}$ is of type $A_{n-1}$. We denote by $\parti_C(2n)$ (resp. $\parti_D(2n)$) the subset of $\parti(2n)$ that corresponds to the nilpotent orbits of type $C$ (resp. type $D$). Similarly, $\parti_B (2n+1)$ denotes the subset of $\parti(2n+1)$ corresponding to the nilpotent orbits of type $B$. We also use $\parti^{ sp}_X(m)$ to denote the partitions of  $m$  corresponding to the special  nilpotent orbits of type $X$.

Let $\mathcal{O}\in \mathcal{N}(\mathfrak{g})$ be a nilpotent orbit and let $C\subset A(\mathcal{O})$ be a conjugacy class. 
Choose $x \in \mathcal{O}$ and let $s \in Z_G(x)$ be a semisimple element whose image in $A(x) \cong A(\mathcal{O})$ belongs to  $C$. 
Set $\mathfrak{l} = Z_{\mathfrak{g}}(s)$. 
We may assume that $\mathfrak{l}$ contains $\operatorname{Lie}(T)$ for some maximal torus $T$; then $\mathfrak{l}$ is a pseudo-Levi subalgebra of $\mathfrak{g}$. 
It is always possible to select $s$ so that $\mathfrak{l}$ has the same semisimple rank as $\mathfrak{g}$, and we shall do so. 
Next we write $\mathfrak{l} = \mathfrak{l}_1 \oplus \mathfrak{l}_2$, where $\mathfrak{l}_2$ is a semisimple subalgebra of the same type as $\mathfrak{g}$ and $\mathfrak{l}_1$ is a simple Lie algebra (possibly zero) that contains the root space corresponding to the lowest root of $\mathfrak{g}$ (the extra node in the extended Dynkin diagram). 
Then $\mathfrak{l}_1$ (if non-zero) is of type $D$ when $\mathfrak{g}$ is of type $B$, of type $C$ when $\mathfrak{g}$ is of type $C$, and of type $D$ when $\mathfrak{g}$ is of type $D$. 
Write $x = x_1 + x_2$ with $x_1 \in \mathfrak{l}_1$ and $x_2 \in \mathfrak{l}_2$. 
Finally, we may adjust the choice of $s$ so that $x_1$ is a distinguished nilpotent element in $\mathfrak{l}_1$. 
Then we can associate with $(\mathcal{O}, C)$ a pair of partitions $(\mathbf{v}, {\mathbf{h} })$, where $\mathbf{v}$ is the partition of $x_1$ in $\mathfrak{l}_1$ and $\mathbf{h}$ is the partition of $x_2$ in $\mathfrak{l}_2$.
We have
\begin{equation} \label{pairs}
\begin{aligned}
\mathbf{v} \in \parti_D (2k), \ & \mathbf{h} \in \parti_B (2n+1-2k) & \quad \text{ in type $B_n$} \\
\mathbf{v} \in \parti_C (2k), \ & \mathbf{h} \in \parti_C (2n-2k) &  \quad \text{ in type $C_n$} \\
\mathbf{v} \in \parti_D (2k), \ & \mathbf{h} \in \parti_D (2n-2k) &  \quad \text{ in type $D_n$} \\
\end{aligned}
\end{equation}
where $0 \leq k \leq n$ and 
$\mathbf{v}$ is a distinguished partition.
The partition $\mathbf{p}$ of $\mathcal{O}$ consists of all the parts in $\mathbf{v}$ and $\mathbf{h}$.
We denote this partition by $\mathbf{p} = \mathbf{v} \cup \mathbf{h}$.

Given $\mathbf{p} =[p_1, \dots,  p_{k-1},  p_k] \in \parti_B (2n+1)$, let $\mathbf{p}^{-} = [p_1, \dots,  p_{k-1}, p_k -1 ]$ and set $\mathbf{p}^C = (\mathbf{p}^{-})_C$.
Then $\mathbf{p}^C \in \parti^{sp}_C(2n)$.
Similarly, given $\mathbf{p} =[p_1,  \dots,  p_k] \in \parti_C (2n)$, let $\mathbf{p}^{+} = [p_1 +1,  \dots, p_{k-1},p_k]$
and
$\mathbf{p}_{+} = [p_1, \dots,  p_{k-1},
 p_k, 1 ]$
and set $\mathbf{p}^B = (\mathbf{p}^{+})_{B}$.
Then $\mathbf{p}^B \in \parti^{sp}_B (2n+1)$ 
and we also have 
\begin{equation}\label{lam-B}
    \mathbf{p}^B = ((\mathbf{p}_{+})^{t}_B)^{t}.
\end{equation}
For $\mathbf{p} \in \parti_B (2n+1)$, we
have $(\mathbf{p}^C)^B = (\mathbf{p}^{t}_B)^{t}$; in particular,
if $\mathbf{p}$ is special, $(\mathbf{p}^C)^B = \mathbf{p}$.
Similarly, for $\mathbf{p} \in \parti_C (2n)$, we
have $(\mathbf{p}^B)^C = (\mathbf{p}^{t}_C)^{t}$; in particular,
if $\mathbf{p}$ is special, $(\mathbf{p}^B)^C = \mathbf{p}$.

\begin{Lem}[{\cite[Lem. 3.3]{ac03}}] \label{pm-col}
    The following identities hold:
    \begin{enumerate}
        \item $((\mathbf{p}^-)_C)^t=(((\mathbf{p})^t)^-)_C$ for $\mathbf{p}\in \parti_B (2n+1)$.
        \item $((\mathbf{p}^+)_B)^t=(((\mathbf{p})^t)^+)_B$ for $\mathbf{p}\in \parti_C (2n)$.
        \item $(((\mathbf{p})^t)_D)^t=((\mathbf{p}^+)^-)_C$ for $\mathbf{p}\in \parti_D (2n)$ or $\mathbf{p}^t\in \parti_C (2n)$.
    \end{enumerate}
\end{Lem}

 We denote by $\mathcal{N}_{o,c}$ the set of pairs $(\mathcal{O}, C)$ 
consisting of an orbit $\mathcal{O} \in \mathcal{N}(\mathfrak{g})$ 
and a conjugacy class $C \subset A(\mathcal{O})$. We have the following result.

\begin{Thm}[{\cite[Thm. 12]{Sommers}}] \label{recipe}
The duality map $d_{\text{Som}}: \mathcal{N}_{o,c} \to { \mathcal{N}(\mathfrak{g}^{\vee})}$ sends the pair 
$(\mathbf{v}, \mathbf{h})$ to the orbit $\mathbf{q}$ according to 
the following recipe:

\begin{equation} \mathbf{q} = \begin{cases}
(\mathbf{v} \cup \mathbf{h}^C)^{t}_C, & \text{$\mathfrak{g}$ type $B$} \\
(\mathbf{v} \cup \mathbf{h}^B)^t_B, & \text{$\mathfrak{g}$ type $C$} \\
(\mathbf{v} \cup (\mathbf{h}^t_D)^t)^{t}_D, & \text{$\mathfrak{g}$ type $D$.} \\
\end{cases}
\end{equation}
\end{Thm}

Note that the case in which $\mathbf{v}$ is the empty partition corresponds to Lusztig--Spaltenstein duality.  In type $D$, if $\mathbf{v}$ is non-empty, our assumptions on  $(\mathbf{v}, \mathbf{h})$ ensure that $\mathbf{h}$ is not very even. 

\subsection{Sommers duality algorithm}

In \cite{BXX}, by applying the RS algorithm to $x\in  \mathrm{Seq}_n (\Gamma)$, we obtain a Young tableau $P(x)$ and a partition ${\bf p}(x)=\mathrm{sh}(P(x))=p(x)=[p_1,p_2,...,p_N]$, where $p_i$ is the number of boxes in the $i$-th row of $P(x)$.  
Recall from \cite{BMW} that we obtain a special partition $H(p(x^-))$ from the given Young tableau $P(x^-)$ or the partition ${\bf p}(x^-)=p(x^-)$ by using the H-algorithm defined in \cite{BMW}. We denote this partition by $p_{X}(x^-)^s$ when it is a special partition of type $X$ (for $X=B,C$ or $D$). Roughly speaking, $H(p(x^-))=p_{X}(x^-)^s$ is the unique  special partition which has the same odd (resp. even) boxes for types $B$ and $C$ (resp. type $D$) as the partition $p(x^-)$.

Recall that  $\mathfrak{g}$ is the Lie algebra of $G$, and let $\mathfrak{h}$ be a Cartan subalgebra.
Let $\Phi^+\subset\Phi$ be the set of positive roots determined by a Borel subalgebra $\mathfrak{b}$ of $\mathfrak{g}$. We have a Cartan decomposition $\mathfrak{g}=\mathfrak{n}\oplus\mathfrak{h}\oplus\mathfrak{n}^-$ with $\mathfrak{b}=\mathfrak{n}\oplus \mathfrak{h}$. Denote by $\Delta$ the set of simple roots in $\Phi^+$.

Assume $\Phi=B_n$, $C_n$ or $D_n$. Fix $\lambda\in\mathfrak{h}^*\simeq\mathbb{C}^n$. Recall that $[\lambda]=(\lambda)_{(0)} \cup (\lambda)_{(\frac{1}{2})}\cup [\lambda]_3$ with $[\lambda]_3=\{{\lambda}_{{Y}_1},\dots,{\lambda}_{{Y}_m}\}$.
Now we describe a decomposition of the root system $ \Phi_{[\lambda]}=\{\alpha\in\Phi\mid\langle\lambda, \alpha^\vee\rangle\in\mathbb{Z}\}\subset\Phi $
into some orthogonal subsystems.

Let $K_{(z)}:=\{i\leq n\mid\lambda_i\in z+\mathbb{Z}\}$ for $z\in\mathbb{C}$.  
For each $ z\in \mathbb{C} $, we define a set $\Phi_{(z)}$  as follows.

If $z\not\in\frac{1}{2}\mathbb{Z}$, we set
\[
\Phi_{(z)}:=\{\ep_i-\ep_j, \pm(\ep_j+\ep_k), \ep_k-\ep_l\mid i, j\in K_{(z)}, k, l\in K_{(-z)}, i\neq j, k\neq l\}.
\] Thus $\Phi_{(z)}$ corresponds to some $\lambda_{Y_i}\in[\lambda]_3$.

If $\Phi=B_n$ and $z=0$ or $\frac{1}{2}$, we set
\begin{equation*}\label{inteq11}
	\Phi_{(z)}:=\{\pm(\ep_i\pm \ep_j), \pm \ep_k\mid i, j, k\in K_{(z)}\ \mbox{and}\ i<j\}.
\end{equation*}

If $\Phi=C_n$ and $z=0$, we set
\begin{equation*}\label{inteq2}
	\Phi_{(0)}:=\{\pm(\ep_i\pm \ep_j), \pm 2\ep_k\mid i, j, k\in K_{(0)}\ \mbox{and}\ i<j\}.
\end{equation*}

If $\Phi=C_n$ and $z=\frac{1}{2}$, we set
\begin{equation*}
	\Phi_{(0.5)}:=\{\pm(\ep_i\pm \ep_j)\mid i, j\in K_{(0.5)}\ \mbox{and}\ i<j\}.
\end{equation*}

If $\Phi=D_n$ and $z=0$ or $\frac{1}{2}$, we set
\begin{equation*}
	\Phi_{(z)}:=\{\pm(\ep_i\pm \ep_j)\mid i, j\in K_{(z)}\ \mbox{and}\ i<j\}.
\end{equation*}

Thus $\Phi_{(0)}$ corresponds to $(\lambda)_{(0)}$ and $\Phi_{(\frac{1}{2})}$ corresponds to  $ (\lambda)_{(\frac{1}{2})}$.

For $z\in\mathbb{C}$, let $K_{(z)}=\{i_1<\cdots<i_r\}$ and $K_{(-z)}=\{j_1<\cdots<j_s\}$. Suppose that $ \Phi_{(z)} \neq \emptyset$.
If $z\not\in\frac{1}{2}\mathbb{Z}$, then $\Phi_{(z)}$ is generated by the simple roots 
\begin{equation*}\label{ninteq1}
\{  \ep_{i_k}-\ep_{i_{k+1}} \mid 1\leq k<r  \}\cup \{ \ep_{i_r}+\ep_{j_{s}}\}\cup \{ \ep_{j_l}-\ep_{j_{l+1}} \mid 1\leq l<s     \},
\end{equation*}
so $\Phi_{(z)}$ is a root system of type $A_{r+s-1}$. If $\Phi=B_n$ (resp. $ D_n $) and $z\in\frac{1}{2}\mathbb{Z}$, $ \Phi_{(z)} $ is a root system of type $ B $ (resp. $ D $). 
If $\Phi=C_n$ and $z\in\mathbb{Z}$ (resp. $ z\in\frac{1}{2}+\mathbb{Z}$), $ \Phi_{(z)} $ is a root system of type $ C$ (resp. $ D $). Note that $ C_1=B_1=A_1 $, $ D_3=A_3 $, $ D_2=A_1\times A_1 $. Hence  $ \Phi_{(z)} $  is irreducible except when it is of type $ D_2 $.

As usual, $\Ree(z)$ denotes the real part of a complex number $z$.

\begin{Prop}[{\cite[Prop. 5.6]{BXX}}]\label{rclem5}
	Fix $\Phi=B_n, C_n$ or $D_n$. For $\lambda\in\mathfrak{h}^*$,
	\[
	\Phi_{[\lambda]}=\bigsqcup_{0\leq\Ree(z)\leq 1/2}\Phi_{(z)}.
	\]
Moreover, $ \Phi_{(z)} $, $ 0\leq\Ree(z)\leq 1/2 $, are mutually orthogonal.

\end{Prop}

Suppose that $\Phi_{[\lambda]}$ decomposes into a direct product of irreducible root systems:
$$\Phi_{[\lambda]}=\Phi_{[\lambda]_1}\times \Phi_{[\lambda]_2}\times \cdots \times \Phi_{[\lambda]_{k+2}}\simeq \Phi_1\times \Phi_2\times \cdots\times \Phi_{k+2} \,,$$ 
where $\Phi_{[\lambda]_1}=\Phi_{(0)}\simeq \Phi_1$, $\Phi_{[\lambda]_2}=\Phi_{(1/2)}\simeq \Phi_2$, and $\Phi_{[\lambda]_j}=\Phi_{(z_j)}\simeq \Phi_{j}$ is a root system of type $A$ for $3\leq j\leq k+2$ and some $z_j \not\in\frac{1}{2}\mathbb{Z}$. We denote such an isomorphism by $\phi$.  
From the isomorphism  $\phi$, we can write 
$$\lam'=\phi(\lam|_{\mathfrak{h}_{\Phi_{[\lambda]}}})=\prod\limits_{1\leq i\leq k+2} \phi(\lambda|_{\mathfrak{h}_{\Phi_{[\lambda]}}})|_{{\Phi_i}}=\prod\limits_{1\leq i\leq k+2} \lam'|_{{\Phi_i}}.$$
Then ${\lam}'|_{\Phi_i}:={\lam}'|_{{\mathfrak{h}_{\Phi_{i}}}}=\phi(\lambda|_{\mathfrak{h}_{\Phi_{[\lambda]_i}}})$ is an integral weight of type $\Phi_i$. In particular ${\lam}'|_{\Phi_1}=(\lambda)_{(0)}$,  ${\lam}'|_{\Phi_2}=(\lambda)_{(\frac{1}{2})}$ and ${\lam}'|_{\Phi_{i+2}}=\lambda_{Y_i}$. 

 From $\lambda_{Y_i}$ for $1\leq i\leq k$, we obtain a special partition ${\bf p}_{i}:=p(\tilde{\lambda}_{Y_i})$ of type $A$ which comes from the Young tableau $P(\tilde{\lambda}_{Y_i})$.
Suppose $\Phi_i$ is of classical type $X_i$ ($X_i=B,C$ or $D$) for $i=1,2$. From $x= (\lambda)_{(0)}$ or $ (\lambda)_{(\frac{1}{2})}$,
 by using the H-algorithm in \cite{BMW}, we obtain two special partitions ${\bf q}_{i}:=p_{X_i}(x^-)^s$ of type $X_i$ and  ${\bf q}^{\vee}_{i}:=p_{X_i^{\vee}}(x^-)^s$ of type $X_i^{\vee}$, both of which come from the Young tableau $P(x^-)$. 

For classical types, we identify a special partition ${\bf p}$ (corresponding to a special nilpotent orbit $\mathcal{O}_w$) with its corresponding special representation $\pi_w=\pi_{\bf p}$ of the Weyl group via the Springer correspondence. We write ${\bf p}^t$ for the transpose of a partition ${\bf p}$. Also, for a partition ${\bf p}$ of classical $X$-type, one has $d_{\rm LS}({\bf p}) = {\bf p}^t_X$.

From now on, we identify a root subsystem  of $\Phi$ with the corresponding analytic subgroup of the given Lie group $G$. Thus $\Phi$ is identified with $G$. We may write $L\subseteq \Phi$ or $\mathcal{O}_L\subseteq \Phi$ to mean that $L$ is a Levi subgroup of $G$ or $\mathcal{O}_L$ is a distinguished orbit for $L$.

From Theorem \ref{som-alg}, we have the following result.

\begin{Thm}\label{som-alge}
    Keep the notation as above. Suppose the root system of $\mathfrak{g}$ is $\Phi$. Let $L(\lam)$ be a highest weight module of $\mathfrak{g}$. Denote
$H:=\Phi_{[\lambda]}\simeq \Phi_1 \times \Phi_2 \times \cdots \times \Phi_{k+2}$ and  $$\lam'=\phi(\lam|_{\mathfrak{h}_{\Phi_{[\lambda]}}})=\prod\limits_{1\leq i\leq k+2} \phi(\lambda|_{\mathfrak{h}_{\Phi_{[\lambda]}}})|_{{\Phi_i}}=\prod\limits_{1\leq i\leq k+2} \lam'|_{{\Phi_i}}.$$ 
Let $X_i$ denote the type of the root system $\Phi_i$ for $i=1,2$.
Then we have the following:
\begin{enumerate}
    \item Suppose that neither ${\bf q}_{1}$ nor ${\bf q}_{2}$ is a very even partitions of type $D$.   Then $$\pi_{w_{\lam}}={\bf q}_{1}\times {\bf q}_{2}\times \prod\limits_{1\leq i\leq k}{\bf p}_{i} \in {\rm Irr}(W(H)), $$ $$\pi^{\vee}_{w_{\lam}}={\bf q}^{\vee}_1\times {\bf q}^{\vee}_2\times \prod\limits_{1\leq i\leq k}{\bf p}_{i} \in {\rm Irr}(W(H^{\vee})), $$ and 
$$d_{\rm LS}^{H^\vee}(\mathcal{O}(\pi^{\vee}_{w_{\lam}}))=({\bf q}^{\vee}_{1})^t_{X_1^{\vee}}\times ({\bf q}^{\vee}_{2})^t_{X_2^{\vee}}\times \prod\limits_{1\leq i\leq k}{\bf p}_{i}^t = {\rm sat}_{L^\vee}^{H^\vee} (\mathcal{O}_{L^\vee}),$$
where $L^\vee \subset \Phi^\vee$ and $\mathcal{O}_{L^\vee}$ is a distinguished orbit of $L^\vee$.

Then  we have
$$\mathcal{O}_{{\rm Ann}(L(\lambda))} = d_{\rm Som}^{G^\vee}(L^{\vee}, \mathcal{O}_{L^\vee}).$$

\item Suppose that $\Phi$ is of type $D$ and both ${\bf q}_{1}$ and ${\bf q}_{2}$ are  very even partitions of type $D$.  Then we have $\pi_{w_{\lam}}=\pi^{\vee}_{w_{\lam}}={\bf q}_{1}\times {\bf q}_{2}\times\prod\limits_{1\leq i\leq k}{\bf p}_{i} \in {\rm Irr}(W(H)) $ and 
$$d_{\rm LS}^{H^\vee}(\mathcal{O}(\pi^{\vee}_{w_{\lam}}))=d_{\rm LS}^{H}(\mathcal{O}(\pi_{w_{\lam}}))=({\bf q}_{1})^t_{D}\times ({\bf q}_{2})^t_{D}\times \prod\limits_{1\leq i\leq k}{\bf p}_{i}^t = {\rm sat}_{L}^{H} (\mathcal{O}_{L}),$$
where $L \subset \Phi$ and $\mathcal{O}_{L}$ is a distinguished orbit of $L$.

Then  we have
$$\mathcal{O}_{{\rm Ann}(L(\lambda))} = d_{\rm Som}^{G^\vee}(L^{\vee}, \mathcal{O}_{L^\vee})={\bf q}_{1}\cupcol {\bf q}_{2}\cupcol(\cupcol_{ i}2{\bf p}_{i}),$$
and the numeral of $\mathcal{O}_{{\rm Ann}(L(\lambda))}$ is determined by the numerals of ${\bf q}_{1}, {\bf q}_{2}$, and $ {\bf p}_{i}$.
\end{enumerate}
\end{Thm}

From Proposition \ref{Orbit-dis}, Theorem \ref{recipe} and \cite[\S 3]{Sommers}, we have the following choice of $\mathbf{v}$ and $\mathbf{h}$ to apply Sommers duality in the above Theorem \ref{som-alge}.

\begin{Cor}\label{type-b-som}
    Let $L(\lam)$ be a highest weight module of $\mathfrak{so}(2n+1,\mathbb{C})$ and
		$[\lambda]=(\lambda)_{(0)} \cup (\lambda)_{(\frac{1}{2})}\cup [\lambda]_3$
		with $[\lambda]_3=\{{\lambda}_{{Y}_1},\dots,{\lambda}_{{Y}_m}\}$.
    Denote
 $H:=\Phi_{[\lambda]}$.

Denote $x_1= (\lambda)_{(0)}$ and $ x_2=(\lambda)_{(\frac{1}{2})}$. For each $1\leq i\leq 2$,
  we have two special partitions ${\bf q}_{i}:=p_{B}(x_i^-)^s$ of type $B$ and  ${\bf q}^{\vee}_{i}:=p_{C}(x_i^-)^s$ of type $C$. For $\lambda_{Y_i}\in [\lambda]_3$, we have a special partition ${\bf p}_{i}:=p(\tilde{\lambda}_{Y_i})$ of type $A$.

Then we have 
   $$\pi_{w_{\lam}}={\bf q}_{1}\times {\bf q}_{2}\times \prod\limits_{1\leq i\leq k}{\bf p}_{i} \in {\rm Irr}(W(H)), $$ $$\pi^{\vee}_{w_{\lam}}={\bf q}^{\vee}_1\times {\bf q}^{\vee}_2\times \prod\limits_{1\leq i\leq k}{\bf p}_{i} \in {\rm Irr}(W(H^{\vee})), $$ 
   and 
$$d_{\rm LS}^{H^\vee}(\mathcal{O}(\pi^{\vee}_{w_{\lam}}))=({\bf q}^{\vee}_{1})^t\times ({\bf q}^{\vee}_{2})^t\times \prod\limits_{1\leq i\leq k}{\bf p}_{i}^t = {\rm sat}_{L^\vee}^{H^\vee} (\mathcal{O}_{L^\vee}),$$
where $L^\vee \subset \Phi^\vee$ and $\mathcal{O}_{L^\vee}$ is a distinguished orbit of $L^\vee$.
Then we have the following.
\begin{enumerate}
    \item When    $({\bf q}^{\vee}_{2})^t$ contains  even parts  with odd multiplicity, let $\mathbf{v}$ denote the partition formed by the largest such even parts of   $({\bf q}^{\vee}_{2})^t$, each appearing exactly once. Let $\mathbf{h}$ denote the remaining parts in $({\bf q}^{\vee}_{1})^t\cup ({\bf q}^{\vee}_{2})^t\cup(\bigcup\limits_{1\leq i\leq k}({\bf p}_{i}^t\cup {\bf p}_{i}^t))$.
    \item When   $({\bf q}^{\vee}_{2})^t$ contains no  even parts  with odd multiplicity, let $\mathbf{v}=\varnothing$ and $\mathbf{h}=({\bf q}^{\vee}_{1})^t\cup ({\bf q}^{\vee}_{2})^t\cup(\bigcup\limits_{1\leq i\leq k}({\bf p}_{i}^t\cup {\bf p}_{i}^t))$.
\end{enumerate}

Then  we have
$$\mathcal{O}_{{\rm Ann}(L(\lambda))} = d_{\rm Som}^{G^\vee}(L^{\vee}, \mathcal{O}_{L^\vee})=d_{\rm{Som}}(\mathbf{v},\mathbf{h})=(\mathbf{v} \cup \mathbf{h}^B)^t_B.$$
\end{Cor}
From the symmetry of the extended Dynkin diagram of type $C$, we can swap the roles
of ${\bf q}^{\vee}_1$ and ${\bf q}^{\vee}_2$. Thus we have the following result.

\begin{Cor}\label{type-b-som-2}
  Keep the notation as above.  Let $L(\lam)$ be a highest weight module of $\mathfrak{so}(2n+1,\mathbb{C})$.
Then we have the following.
\begin{enumerate}
    \item When    $({\bf q}^{\vee}_{1})^t$ contains  even parts  with odd multiplicity, let $\mathbf{v}$ denote the partition formed by the largest such even parts of   $({\bf q}^{\vee}_{1})^t$, each appearing exactly once. Let $\mathbf{h}$ denote the remaining parts in $({\bf q}^{\vee}_{1})^t\cup ({\bf q}^{\vee}_{2})^t\cup(\bigcup\limits_{1\leq i\leq k}({\bf p}_{i}^t\cup {\bf p}_{i}^t))$.
    \item When   $({\bf q}^{\vee}_{1})^t$ contains no  even parts  with odd multiplicity, let $\mathbf{v}=\varnothing$ and $\mathbf{h}=({\bf q}^{\vee}_{1})^t\cup ({\bf q}^{\vee}_{2})^t\cup(\bigcup\limits_{1\leq i\leq k}({\bf p}_{i}^t\cup {\bf p}_{i}^t))$.
\end{enumerate}

Then  we have
$$\mathcal{O}_{{\rm Ann}(L(\lambda))} = d_{\rm Som}^{G^\vee}(L^{\vee}, \mathcal{O}_{L^\vee})=d_{\rm{Som}}(\mathbf{v},\mathbf{h})=(\mathbf{v} \cup \mathbf{h}^B)^t_B.$$

\end{Cor}

 \begin{Rem}
    Note that we have  ${\bf q}^{\vee}_{i}={\bf q}^{C}_{i}$ and  $({\bf q}^{\vee}_{i})^B={\bf q}_{i}$.
 \end{Rem}

\begin{Cor}\label{type-c-som}
Let $L(\lam)$ be a highest weight module of $\mathfrak{sp}(2n,\mathbb{C})$ and
		$[\lambda]=(\lambda)_{(0)} \cup (\lambda)_{(\frac{1}{2})}\cup [\lambda]_3$
				with $[\lambda]_3=\{{\lambda}_{{Y}_1},\dots,{\lambda}_{{Y}_m}\}$.
    Denote
 $H:=\Phi_{[\lambda]}$.

For $x= (\lambda)_{(0)}$,
  we have two special partitions ${\bf q}_{1}:=p_{C}(x^-)^s$ of type $C$ and  ${\bf q}^{\vee}_{1}:=p_{B}(x^-)^s$ of type $B$. For $y=(\lambda)_{(\frac{1}{2})}$,
  we have a special partition ${\bf q}_{2}={\bf q}^{\vee}_{2}:=p_{D}(y^-)^s$ of type $D$. For $\lambda_{Y_i}\in [\lambda]_3$, we have a special partition ${\bf p}_{i}:=p(\tilde{\lambda}_{Y_i})$ of type $A$.

Then we have 
   $$\pi_{w_{\lam}}={\bf q}_{1}\times {\bf q}_{2}\times \prod\limits_{1\leq i\leq k}{\bf p}_{i} \in {\rm Irr}(W(H)), $$ $$\pi^{\vee}_{w_{\lam}}={\bf q}^{\vee}_1\times {\bf q}_2\times \prod\limits_{1\leq i\leq k}{\bf p}_{i} \in {\rm Irr}(W(H^{\vee})), $$ 
   and 
$$d_{\rm LS}^{H^\vee}(\mathcal{O}(\pi^{\vee}_{w_{\lam}}))=({\bf q}^{\vee}_{1})^t\times ({\bf q}_{2}^t)_D\times \prod\limits_{1\leq i\leq k}{\bf p}_{i}^t = {\rm sat}_{L^\vee}^{H^\vee} (\mathcal{O}_{L^\vee}),$$
where $L^\vee \subset \Phi^\vee$ and $\mathcal{O}_{L^\vee}$ is a distinguished orbit of $L^\vee$.
Then we have the following.
\begin{enumerate}
    \item When   $({\bf q}_{2}^t)_D$  contains odd parts with odd multiplicity, let $\mathbf{v}$ denote the partition formed by the maximal such odd parts of  $({\bf q}_{2}^t)_D$, each appearing exactly once. Let $\mathbf{h}$ denote the remaining parts in $({\bf q}^{\vee}_{1})^t\cup ({\bf q}_{2}^t)_D\cup(\bigcup\limits_{1\leq i\leq k}({\bf p}_{i}^t\cup {\bf p}_{i}^t))$.
    \item When  $({\bf q}_{2}^t)_D$ contains no odd parts  with odd multiplicity, let $\mathbf{v}=\varnothing$ and $\mathbf{h}=({\bf q}^{\vee}_{1})^t\cup ({\bf q}_{2}^t)_D\cup(\bigcup\limits_{1\leq i\leq k}({\bf p}_{i}^t\cup {\bf p}_{i}^t))$.
\end{enumerate}

Then  we have
$$\mathcal{O}_{{\rm Ann}(L(\lambda))} = d_{\rm Som}^{G^\vee}(L^{\vee}, \mathcal{O}_{L^\vee})=d_{\rm{Som}}(\mathbf{v},\mathbf{h})=(\mathbf{v} \cup \mathbf{h}^C)^t_C.$$

\end{Cor}

\begin{Rem}
    Note that we have  ${\bf q}^{\vee}_{1}={\bf q}^{B}_{1}$ and  $({\bf q}^{\vee}_{1})^C={\bf q}_{1}$.
 \end{Rem}

\begin{Cor}\label{type-D-som}
Let $L(\lam)$ be a highest weight module of $\mathfrak{so}(2n,\mathbb{C})$ and
		$[\lambda]=(\lambda)_{(0)} \cup (\lambda)_{(\frac{1}{2})}\cup [\lambda]_3$
				with $[\lambda]_3=\{{\lambda}_{{Y}_1},\dots,{\lambda}_{{Y}_m}\}$.
    Denote
 $H:=\Phi_{[\lambda]}$.

Denote $x_1= (\lambda)_{(0)}$ and $ x_2=(\lambda)_{(\frac{1}{2})}$. For each $1\leq i\leq 2$,
  we have a special partition ${\bf q}^{\vee}_{i}={\bf q}_{i}:=p_{D}(x_i^-)^s$ of type $D$. For $\lambda_{Y_i}\in [\lambda]_3$, we have a special partition ${\bf p}_{i}:=p(\tilde{\lambda}_{Y_i})$ of type $A$.

Then we have 
   $$\pi_{w_{\lam}}=\pi^{\vee}_{w_{\lam}}={\bf q}_{1}\times {\bf q}_{2}\times \prod\limits_{1\leq i\leq k}{\bf p}_{i} \in {\rm Irr}(W(H)) $$ 
   and
$$d_{\rm LS}^{H^\vee}(\mathcal{O}(\pi^{\vee}_{w_{\lam}}))=({\bf q}_{1}^t)_D\times ({\bf q}_{2}^t)_D\times \prod\limits_{1\leq i\leq k}{\bf p}_{i}^t = {\rm sat}_{L^\vee}^{H^\vee} (\mathcal{O}_{L^\vee}),$$
where $L^\vee \subset \Phi^\vee$ and $\mathcal{O}_{L^\vee}$ a distinguished orbit of $L^\vee$.
Then we have the following.
\begin{enumerate}
    \item When   $({\bf q}_{2}^t)_D$  contains  odd parts  with odd multiplicity, let $\mathbf{v}$ denote the partition formed by the maximal such odd parts of  $({\bf q}_{2}^t)_D$, each appearing exactly once. Let $\mathbf{h}$ denote the remaining parts in $({\bf q}_{1}^t)_D\cup ({\bf q}_{2}^t)_D\cup(\bigcup\limits_{1\leq i\leq k}({\bf p}_{i}^t\cup {\bf p}_{i}^t))$.
    \item When $({\bf q}_{2}^t)_D$ contains no odd parts  with odd multiplicity, let $\mathbf{v}=\varnothing$ and $\mathbf{h}=({\bf q}_{1}^t)_D\cup ({\bf q}_{2}^t)_D\cup(\bigcup\limits_{1\leq i\leq k}({\bf p}_{i}^t\cup {\bf p}_{i}^t))$.
\end{enumerate}

Then  we have
$$\mathcal{O}_{{\rm Ann}(L(\lambda))} = d_{\rm Som}^{G^\vee}(L^{\vee}, \mathcal{O}_{L^\vee})=d_{\rm{Som}}(\mathbf{v},\mathbf{h})=(\mathbf{v} \cup (\mathbf{h}^t_D)^t)^t_D.$$
\end{Cor}

From the symmetry of the extended Dynkin diagram of type $D$, we can swap the roles
of ${\bf q}_1$ and ${\bf q}_2$. Thus we have the following result.

\begin{Cor}\label{type-D-som-2}
Keep the notation as above. Let $L(\lam)$ be a highest weight module of $\mathfrak{so}(2n,\mathbb{C})$.
Then we have the following.
\begin{enumerate}
    \item When   $({\bf q}_{1}^t)_D$  contains  odd parts  with odd multiplicity, let $\mathbf{v}$ denote the partition formed by the maximal such odd parts of  $({\bf q}_{1}^t)_D$, each appearing exactly once. Let $\mathbf{h}$ denote the remaining parts in $({\bf q}_{1}^t)_D\cup ({\bf q}_{2}^t)_D\cup(\bigcup\limits_{1\leq i\leq k}({\bf p}_{i}^t\cup {\bf p}_{i}^t))$.
    \item When $({\bf q}_{1}^t)_D$ contains no odd parts  with odd multiplicity, let $\mathbf{v}=\varnothing$ and $\mathbf{h}=({\bf q}_{1}^t)_D\cup ({\bf q}_{2}^t)_D\cup(\bigcup\limits_{1\leq i\leq k}({\bf p}_{i}^t\cup {\bf p}_{i}^t))$.
\end{enumerate}

Then  we have
$$\mathcal{O}_{{\rm Ann}(L(\lambda))} = d_{\rm Som}^{G^\vee}(L^{\vee}, \mathcal{O}_{L^\vee})=d_{\rm{Som}}(\mathbf{v},\mathbf{h})=(\mathbf{v} \cup (\mathbf{h}^t_D)^t)^t_D.$$

\end{Cor}
\begin{Rem}
    In the above Corollary \ref{type-D-som}, $\mathcal{O}_{{\rm Ann}(L(\lambda))}$ is a very even orbit if and only if $({\bf q}_{1}^t)_D$  and $({\bf q}_{2}^t)_D$  have no odd parts. In this case, we have 
    $$\mathcal{O}_{{\rm Ann}(L(\lambda))} = d_{\rm Som}^{G^\vee}(L^{\vee}, \mathcal{O}_{L^\vee})=\mathbf{h}^t_D={\bf q}_{1}\cupcol {\bf q}_{2}\cupcol(\cupcol_{ i}2{\bf p}_{i}).$$
    
    The numeral of $\mathcal{O}_{{\rm Ann}(L(\lambda))}$ can be determined by Theorem \ref{pa-suanfa}.
\end{Rem}

\section{The new proof for the partition algorithm by the Sommers duality} \label{new-proof}  

In this section, we give a new proof of the partition algorithm in Theorem \ref{pa-suanfa}. Our proof is based on the Sommers duality. 

First, we establish several useful lemmas needed  before we give the proof of Theorem \ref{pa-suanfa}.

\begin{Lem}\label{lem:collapse-even-shift}
Let ${\bf a}\in \parti (2m+1)$ and ${\bf x}\in \parti(2m)$, and let ${\bf b}$ be a partition all of whose parts are even.
Then the following identities hold:
\[
({\bf a} \csqcup {\bf b})_B = ({\bf a}_B \csqcup {\bf b})_B, ({\bf x} \csqcup {\bf b})_C = ({\bf x}_C \csqcup {\bf b})_C, ({\bf x} \csqcup {\bf b})_D = ({\bf x}_D \csqcup {\bf b})_D.
\]
\end{Lem}

\begin{proof}
Recall that the $B$-collapse of a partition is obtained by repeatedly replacing the largest even part that occurs an odd number of times: if $2k$ is such a part, the last occurrence of $2k$ is replaced by $2k-1$, and the first subsequent even part strictly smaller than $2k-1$ is increased by $1$; this process continues until the resulting partition is a partition of type $B$.
All parts of ${\bf b}$ are even, so adding ${\bf b}$ (row‑wise) shifts the parts of ${\bf a}$ by even quantities. Consequently, any odd part of ${\bf a}$ is still  an odd part of ${\bf a} \csqcup {\bf b}$.  Suppose $a_{i_0}$ is an odd part of ${\bf a}$. Then $a_{i_0}+b_{i_0}$ is still an odd part in ${\bf a} \csqcup {\bf b}$ and ${\bf a}_B \csqcup {\bf b}$ since the boxes in odd parts will not be moved during the $B$-collapse. Thus $a_{i_0}+b_{i_0}$ is still an odd part in $({\bf a} \csqcup {\bf b})_B$ and $({\bf a}_B \csqcup {\bf b})_B$.
So we may assume that there is only one odd part in ${\bf a}$ and this odd part is the first part of ${\bf a}$.

Assume $\mathbf{a}$ is not of type $B$.
Let $2k$ be the largest even part of $\mathbf{a}$ that occurs an odd number of times,
and let $p$ be the index of the last row of $\mathbf{a}$ that equals $2k$. Then $p$ is an even number. 
One $B$‑collapse step replaces $a_p$ by $2k-1$ and, for the first index $q>p$ with $a_q$ being even and
$a_q < 2k-1$, replaces $a_q$ by $a_q+1$; if no such $q$ exists, a new part $1$ is
appended. Let $\mathbf{a}'$ denote this new partition obtained from $\mathbf{a}$ by one $B$‑collapse step.

We prove the lemma by induction on the number of elementary collapse steps
needed to turn $\mathbf{a}$ into $\mathbf{a}_B$.
It suffices to prove
\begin{equation}\label{step-1}
(\mathbf{a}' \csqcup \mathbf{b})_B = (\mathbf{a} \csqcup \mathbf{b})_B ,
\end{equation}
for then applying this identity successively yields our result in the lemma.

Write $\mathbf{a} = [a_1,\dots,a_r]$ and $\mathbf{b} = [b_1,\dots,b_r]$ (padding the shorter one with zeros).
We have the following two partitions:
\begin{align*}
\mathbf{a} \csqcup \mathbf{b} &= [a_1+b_1,  \dots, a_{p}+b_{p},\dots, a_q+b_q,  \dots,  a_r+b_r],\\
\mathbf{a}' \csqcup \mathbf{b} &= [a_1+b_1,  \dots, a_{p}-1+b_{p},\dots, a_q+1+b_q,   \dots,  a_r+b_r].
\end{align*}
Let $\mathbf{u} = \mathbf{a} \csqcup \mathbf{b}$ and $\mathbf{v} = \mathbf{a}' \csqcup \mathbf{b}$.

Denote $\mathbf{u}'=[a_1+b_1,  \dots, a_{p}+b_{p}]$ and $\mathbf{v}'=[a_1+b_1,  \dots, a_{p}-1+b_{p}]$. Then we write $\mathbf{u}=\mathbf{u}'\cup \mathbf{u}''$
and $\mathbf{v}=\mathbf{v}'\cup \mathbf{v}''$. From \cite[Lem. 3.1]{ac03}, we have $\mathbf{u}_B=((\mathbf{u}')^-)_D\cup ((\mathbf{u}'')^+)_B $ and $\mathbf{v}_B=(\mathbf{v}')_D\cup (\mathbf{v}'')_B$ since $\mathbf{u}'$ and $\mathbf{v}'$ both have $p$ (an even number) parts and $|\mathbf{u}'|=|\mathbf{v}'|+1$ is odd.
Note that $(\mathbf{u}')^-=\mathbf{v}'$. Thus we only need to prove that $((\mathbf{u}'')^+)_B=(\mathbf{v}'')_B$, that is,
$$[a_{p+1}+b_{p+1}+1,\dots, a_q+b_q,  \dots,  a_r+b_r]_B=[a_{p+1}+b_{p+1},\dots, a_q+1+b_q,  \dots,  a_r+b_r]_B.$$
Recall that $\mathbf{a}$ has only even parts except for the first part. Now $q$ is the first index such that $q>p$ and $a_q < a_p-1$.
If $p<r$, then $a_p$ has an even row part below it; i.e., $a_{p+1}$ is an even integer smaller than $a_p$. Thus we have $q=p+1$ and we only need to prove that $$[a_{q}+b_{q}+1,a_{q+1}+b_{q+1},\dots,   a_r+b_r]_B=[a_{q}+1+b_{q},a_{q+1}+b_{q+1},\dots,   a_r+b_r]_B,$$
which is trivial.
If $p=r$, we can append $a_{r+1}=b_{r+1}=0$ to $\mathbf{a}$ and $\mathbf{b}$. Thus we still have $q=r+1=p+1$. The remaining arguments are identical.

Hence we have proved that
\[
(\mathbf{a} \csqcup \mathbf{b})_B = (\mathbf{a}' \csqcup \mathbf{b})_B.
\]

By induction, we obtain
\[
(\mathbf{a} \csqcup \mathbf{b})_B = (\mathbf{a}_B \csqcup \mathbf{b})_B ,
\]
as required.

The arguments for types $C$ and $D$ are similar to those for type $B$.
\end{proof}

\begin{Lem}\label{lem:2}
Let ${\bf b}\in \parti(2m)$ such that all of its columns are even and have multiplicity one. Let $\mathbf{a}= \mathbf{s} \csqcup \mathbf{e}\in \parti(2k+1)$, where  
$\mathbf{s} \in \parti_B^{sp}(2k'+1)$ for some integer $k'\leq k$, and $\mathbf{e}$ is a partition every column of which appears with an even multiplicity.
Then we have
\begin{equation}\label{eq:main}
    (\mathbf{a}_B \csqcup \mathbf{b})_B = (\mathbf{a} \csqcup \mathbf{b})_B .
\end{equation}

\end{Lem}

\begin{proof}

First we note that all row parts of $\mathbf{e}$ are even.
Because $\mathbf{s}$ is a special  partition of type $B$, its even parts already occur with even
multiplicity; an even part of $\mathbf{a}$ that appears an odd number of times can only be created
by the interaction with the even rows of $\mathbf{e}$, but adding an even number to each part
of $\mathbf{s}$ does not change the parity of the multiplicities of the shifted parts.
Consequently, the ``offending'' even parts of $\mathbf{a}$ (those with odd multiplicity)
all come from the corresponding even parts of $\mathbf{s}$ after translation by the even
row parts of $\mathbf{e}$. Since $\mathbf{s}$ is a special partition of type $B$ and all row parts of $\mathbf{e}$ are even, $\mathbf{a}$ has an even number of odd parts between any consecutive even parts and an odd number of odd parts greater than the largest even part. Thus the first row part of $\mathbf{a}$ is an odd number. 

The condition on $\mathbf{b}$ implies that its transpose
has all parts even and distinct; therefore the rows of $\mathbf{b}$ satisfy
\begin{equation}\label{b-pair}
b_{2i-1}=b_{2i}\qquad (i=1,2,\dots),
\end{equation}
and we can write $\mathbf{b}=[m,m,...,m,m,m-1,m-1,...,m-1,m-1,...,2,2,...,2,2,1,1,...,1,1]$, where each integer $1\leq l\leq m$ occurs with an even multiplicity.

We prove the lemma by induction on the number of elementary collapse steps
needed to turn $\mathbf{a}$ into $\mathbf{a}_B$.
Let $\mathbf{a}'$ be obtained from $\mathbf{a}$ by one $B$‑collapse step.
It suffices to prove
\begin{equation}\label{step}
(\mathbf{a}' \csqcup \mathbf{b})_B = (\mathbf{a} \csqcup \mathbf{b})_B ,
\end{equation}
for then applying this identity successively yields~\eqref{eq:main}.

Assume $\mathbf{a}$ is not of type $B$.
Let $2k$ be the largest even part of $\mathbf{a}$ that occurs an odd number of times,
and let $p$ be the index of the last row of $\mathbf{a}$ that equals $2k$. Then $p$ is an even number. 
One $B$‑collapse step replaces $a_p$ by $2k-1$ and, for the first index $q>p$ with $a_q$ being even and
$a_q < 2k-1$, replaces $a_q$ by $a_q+1$; if no such $q$ exists, a new part $1$ is
appended. We use $\mathbf{a}'$ to denote this new partition obtained from $\mathbf{a}$ by one $B$‑collapse step.

Write $\mathbf{a} = [a_1,\dots,a_r]$ and $\mathbf{b} = [b_1,\dots,b_r]$ (padding the shorter one with zeros).
We have the following two partitions:
\begin{align*}
\mathbf{a} \csqcup \mathbf{b} &= [a_1+b_1,  \dots, a_{p}+b_{p},\dots, a_q+b_q,  \dots,  a_r+b_r],\\
\mathbf{a}' \csqcup \mathbf{b} &= [a_1+b_1,  \dots, a_{p}-1+b_{p},\dots, a_q+1+b_q,   \dots,  a_r+b_r].
\end{align*}
Let $\mathbf{u} = \mathbf{a} \csqcup \mathbf{b}$ and $\mathbf{v} = \mathbf{a}' \csqcup \mathbf{b}$.

Recall that there is an odd number (denote it by $k_0$) of odd parts
greater than the largest even part of  $\mathbf{a}$. Suppose $k_0>1$. After translation by an equal pair of
row parts of $\mathbf{b}$, we obtain a pair of odd parts when $b_1$ is even, and a pair of even parts when $b_1$ is odd. Consequently, after the $B$-collapse, the first two row parts of $\mathbf{u}_B$ and $\mathbf{v}_B$ coincide; they become either $a_1+b_1$ and $a_2+b_2$ or $a_1+b_1-1$ and $a_2+b_2+1$. Thus no box is moved down to the third row or below. Equivalently, we can remove the first two row parts of $\mathbf{u}$ and $\mathbf{v}$ and denote them by $\mathbf{u}_{1,2}=\mathbf{v}_{1,2}=[a_1+b_1,a_2+b_2]$ and denote the remaining partitions of $\mathbf{u}$ and $\mathbf{v}$ by $\mathbf{u}_l$ and $\mathbf{v}_l$, respectively.
Then we have $$\mathbf{u}_B=(\mathbf{u}_{1,2})_D\cup (\mathbf{u}_l)_B$$ and $$\mathbf{v}_B=(\mathbf{v}_{1,2})_D\cup (\mathbf{v}_l)_B.$$
Therefore by induction, we may assume that $k_0=1$.

Denote $\mathbf{u}'=[a_1+b_1,  \dots, a_{p}+b_{p}]$ and $\mathbf{v}'=[a_1+b_1,  \dots, a_{p}-1+b_{p}]$. Then we write $\mathbf{u}=\mathbf{u}'\cup \mathbf{u}''$
and $\mathbf{v}=\mathbf{v}'\cup \mathbf{v}''$. From \cite[Lem. 3.1]{ac03}, we have $\mathbf{u}_B=((\mathbf{u}')^-)_D\cup ((\mathbf{u}'')^+)_B $ and $\mathbf{v}_B=(\mathbf{v}')_D\cup (\mathbf{v}'')_B$ since $\mathbf{u}'$ and $\mathbf{v}'$ both have $p$ (an even number) parts and $|\mathbf{u}'|=|\mathbf{v}'|+1$ is odd.
Note that $(\mathbf{u}')^-=\mathbf{v}'$. Thus we only need to prove that $((\mathbf{u}'')^+)_B=(\mathbf{v}'')_B$, that is,
$$[a_{p+1}+b_{p+1}+1,\dots, a_q+b_q,  \dots,  a_r+b_r]_B=[a_{p+1}+b_{p+1},\dots, a_q+1+b_q,  \dots,  a_r+b_r]_B.$$
Recall that $\mathbf{a}$ has an even number of odd parts between any consecutive even parts. Thus $\mathbf{a}$ has an even number of even parts between any consecutive odd parts and it can only have an odd number of even parts below  the last odd part. Now $q$ is the first index such that $q>p$ and $a_q < a_p-1$.
If $p<r$, $a_p$ has an even row part below it; i.e., $a_{p+1}$ is an even integer smaller than $a_p$. Thus we have $q=p+1$ and we only need to prove that $$[a_{q}+b_{q}+1,a_{q+1}+b_{q+1},\dots,   a_r+b_r]_B=[a_{q}+1+b_{q},a_{q+1}+b_{q+1},\dots,   a_r+b_r]_B,$$
which is trivial.
If $p=r$, we can append $a_{r+1}=b_{r+1}=0$ to $\mathbf{a}$ and $\mathbf{b}$. Thus we still have $q=r+1=p+1$. The remaining arguments are identical.

Hence we have proved that
\[
(\mathbf{a} \csqcup \mathbf{b})_B = (\mathbf{a}' \csqcup \mathbf{b})_B ,
\]
which is~\eqref{step}.

By induction, we obtain
\[
(\mathbf{a} \csqcup \mathbf{b})_B = (\mathbf{a}_B \csqcup \mathbf{b})_B ,
\]
as required.
\end{proof}

\begin{Rem}
The condition $\mathbf{a} = \mathbf{s} \csqcup \mathbf{e}$ cannot be dropped.
For example, take $\mathbf{b}$ with $\mathbf{b}^t = [6,4,2]$ (so $\mathbf{b}=[3,3,2,2,1,1]$) and $\mathbf{a}=[6,3]$.
Here $\mathbf{a}$ fails to be of the required form.
One computes
\[
(\mathbf{a}_B \csqcup \mathbf{b})_B = [7,7,3,1,1,1,1],\qquad
(\mathbf{a} \csqcup \mathbf{b})_B = [9,5,3,1,1,1,1],
\]
and the two sides are not equal.
\end{Rem}

\begin{Lem}\label{lem-C}
Let ${\bf b}\in \parti(2m)$ such that all of its columns are odd and have multiplicity one. Let $\mathbf{a}= \mathbf{s} \csqcup \mathbf{e}\in \parti(2k)$ and $\mathbf{x}=\mathbf{t} \csqcup \mathbf{e}\in \parti(2k)$, where  
$\mathbf{s} \in \parti_C^{sp}(2k')$ and $\mathbf{t} \in \parti_D^{sp}(2k')$ for some integer $k'\leq k$, respectively, and $\mathbf{e}$ is a partition every column of which appears with an even multiplicity. 
Then we have
\[
    ({\bf a}_C \csqcup {\bf b})_C = ({\bf a} \csqcup {\bf b})_C \quad\text{and}\quad ({\bf x}_D \csqcup {\bf b})_D = ({\bf x} \csqcup {\bf b})_D.
\]

\end{Lem}
\begin{proof}
   The proof is analogous to that of Lemma \ref{lem:2} and is therefore omitted.
\end{proof}

\begin{Lem}\label{C-empty}
Let ${\bf a}$ be a special partition of type $B$ and let ${\bf b}$ be a partition all of whose parts are even.
Then
\[
({\bf a} \csqcup {\bf b})^C = ({\bf a}^C \csqcup {\bf b})_C.
\]
\end{Lem}

\begin{proof}

First we claim that the following holds:
\begin{equation}\label{eq:shift}
\bigl((\mathbf{a} \csqcup \mathbf{b})^-\bigr)_C = (\mathbf{a}^- \csqcup \mathbf{b})_C .
\end{equation}

Write $\mathbf{a} = [a_1,\dots,a_r]$ and $\mathbf{b} = [b_1,\dots,b_r]$ (padding the shorter one with zeros).

When ${\bf b}$ is shorter, we obviously have  $$(\mathbf{a} \csqcup \mathbf{b})^-=\mathbf{a}^- \csqcup \mathbf{b}.$$ 
Thus  equation (\ref{eq:shift}) follows from this identity.

Now assume that $\mathbf{a}$ is shorter. 
Let $k$ be the index of the last positive part of $\mathbf{a}$, so $a_k > 0$ and $a_i = 0$ for $i > k$.
  The two partitions are
\begin{align*}
\mathbf{a} \csqcup \mathbf{b} &= [a_1+b_1,  \dots, a_{k-1}+b_{k-1},  a_k+b_k,  b_{k+1},  \dots,  b_r],\\
\mathbf{a}^- \csqcup \mathbf{b} &= [a_1+b_1,  \dots, a_{k-1}+b_{k-1},  a_k-1+b_k,  b_{k+1},  \dots,  b_r],
\end{align*}
while $(\mathbf{a} \csqcup \mathbf{b})^-$ is obtained from the first sequence by subtracting $1$ from its \emph{last} part $b_r$ (if $a_k+b_k$ is not the last part). 
Thus $(\mathbf{a} \csqcup \mathbf{b})^-$ and $\mathbf{a}^- \csqcup \mathbf{b}$ differ in exactly two places: one has a  `$-1$' at position $r$, the other has a `$-1$' at position $k$.

Let $\mathbf{u} = (\mathbf{a} \csqcup \mathbf{b})^-$ and $\mathbf{v} = \mathbf{a}^- \csqcup \mathbf{b}$.
Observe that $\mathbf{u}$ and $\mathbf{v}$ have the same list of parts except that $\mathbf{u}$ contains a part $b_r-1$ and a part $a_k+b_k$, while $\mathbf{v}$ contains $b_r$ and $a_k-1+b_k$. 
Because every $b_i$ is even, the parity of the affected entries changes exactly by $1$, and all other parts keep their parity.
The $C$‑collapse algorithm repeatedly reduces the largest odd part that occurs an odd number of times.
Note that the extra odd part $b_r-1$ appearing in $\mathbf{u}$  is the largest offending odd part, and it will be changed to $b_r$ after the final $C$-collapse. Thus there is a unit   `$1$' coming from some previous odd part of $\mathbf{u}$. If it comes from $a_k+b_k$, we will have $\mathbf{u}_C=[\cdots,a_k-1+b_k,b_{k+1},\cdots,b_r]=\mathbf{v}_C$ since the last $r-k$ parts of $\mathbf{u}_C$ and $\mathbf{v}$ are the same.
If the unit `$1$' comes from some odd number $a_{k'}+b_{k'}>a_k+b_k$ for some $k'<k$, then $a_k+b_k$ will be an even number. So $a_k-1+b_k$ will be an odd number and will occur with multiplicity one in ${\bf v}$.  After several steps of $C$-collapse, ${\bf v}$ will be changed to some ${\bf v}'=[\cdots,a_{k'}+b_{k'},\cdots,a_k-1+b_k,b_{k+1},\cdots,b_r]$. Now $a_{k'}+b_{k'}$ and $a_k-1+b_k$ are both odd numbers and occur with odd multiplicity. We continue the $C$-collapse for ${\bf v}'$ and will get ${\bf v}_C=[\cdots,a_{k'}+b_{k'}-1,\cdots,a_k+b_k,b_{k+1},\cdots,b_r]={\bf u}_C$ since the last $r-k$ parts of $\mathbf{v}_C$ and $\mathbf{u}$ are the same.

Using Lemma \ref{lem:collapse-even-shift} with $\mathbf{x} = \mathbf{a}^-$ we have
\[
(\mathbf{a}^- \csqcup \mathbf{b})_C = \bigl((\mathbf{a}^-)_C \csqcup \mathbf{b}\bigr)_C = (\mathbf{a}^C \csqcup \mathbf{b})_C .
\]
On the other hand, by definition $(\mathbf{a} \csqcup \mathbf{b})^C = \bigl((\mathbf{a} \csqcup \mathbf{b})^-\bigr)_C$, which by \eqref{eq:shift} equals $(\mathbf{a}^- \csqcup \mathbf{b})_C$.
Combining the equalities yields
\[
(\mathbf{a} \csqcup \mathbf{b})^C = (\mathbf{a}^C \csqcup \mathbf{b})_C .
\]
\end{proof}

\subsection{Proof for type $B_n$}
Let $L(\lam)$ be a highest weight module of $\mathfrak{so}(2n+1,\mathbb{C})$ and
		$[\lambda]=(\lambda)_{(0)} \cup (\lambda)_{(\frac{1}{2})}\cup [\lambda]_3$
		with $[\lambda]_3=\{{\lambda}_{{Y}_1},\dots,{\lambda}_{{Y}_m}\}$. Denote
 $H:=\Phi_{[\lambda]}$.

Denote $x_1= (\lambda)_{(0)}$ and $ x_2=(\lambda)_{(\frac{1}{2})}$. For each $1\leq i\leq 2$,
  we have two special partitions ${\bf q}_{i}:=p_{B}(x_i^-)^s$ of type $B$ and  ${\bf q}^{\vee}_{i}:=p_{C}(x_i^-)^s$ of type $C$. For $\lambda_{Y_i}\in [\lambda]_3$, we have a special partition ${\bf p}_{i}:=p(\tilde{\lambda}_{Y_i})$ of type $A$.

Then we have 
   $$\pi_{w_{\lam}}={\bf q}_{1}\times {\bf q}_{2}\times \prod\limits_{1\leq i\leq k}{\bf p}_{i} \in {\rm Irr}(W(H)), $$ $$\pi^{\vee}_{w_{\lam}}={\bf q}^{\vee}_1\times {\bf q}^{\vee}_2\times \prod\limits_{1\leq i\leq k}{\bf p}_{i} \in {\rm Irr}(W(H^{\vee})), $$ 
   and 
$$d_{\rm LS}^{H^\vee}(\mathcal{O}(\pi^{\vee}_{w_{\lam}}))=({\bf q}^{\vee}_{1})^t\times ({\bf q}^{\vee}_{2})^t\times \prod\limits_{1\leq i\leq k}{\bf p}_{i}^t = {\rm sat}_{L^\vee}^{H^\vee} (\mathcal{O}_{L^\vee}),$$
where $L^\vee \subset \Phi^\vee$ and $\mathcal{O}_{L^\vee}$ is a distinguished orbit of $L^\vee$.

Now we give our proof for the type $B$.

\begin{enumerate}
    \item When $({\bf q}^{\vee}_{2})^t$ contains no  even parts  with odd multiplicity, let $\mathbf{v}=\varnothing$ and $\mathbf{h}=({\bf q}^{\vee}_{1})^t\cup ({\bf q}^{\vee}_{2})^t\cup(\bigcup\limits_{1\leq i\leq k}({\bf p}_{i}^t\cup {\bf p}_{i}^t))$. Then from Corollary \ref{type-b-som}, we have
\begin{align*}\mathcal{O}_{{\rm Ann}(L(\lambda))} &= d_{\rm Som}^{C}(L^{\vee}, \mathcal{O}_{L^\vee})=d_{\text{Som}}(\mathbf{v},\mathbf{h})=(\mathbf{v} \cup \mathbf{h}^B)^t_B=(\mathbf{h}^B)^t_B   \\
&=(((\mathbf{h}_{+})^{t}_B)^{t})^t_B  \quad \quad\quad\quad\quad \text{by (\ref{lam-B})}\\
&=(\mathbf{h}_{+})^{t}_B\\
&=\left( ({\bf q}^{\vee}_{1})^t\cup ({\bf q}^{\vee}_{2})^t\cup(\bigcup\limits_{1\leq i\leq k}({\bf p}_{i}^t\cup {\bf p}_{i}^t))\cup [1]   \right)^t_B\\
&=\left({\bf q}^{\vee}_{1}\csqcup {\bf q}^{\vee}_{2}\csqcup \left(\csqcup_{i}2{\bf p}_{i}\right)\csqcup [1]\right)_B\\
&=\left(({\bf q}^{\vee}_{1}\csqcup [1])\csqcup {\bf q}^{\vee}_{2}\csqcup \left(\csqcup_{i}2{\bf p}_{i}\right)\right)_B.
\end{align*}
Now $\csqcup_{i}2{\bf p}_{i}$ consists of even parts. In
addition, each part of  $({\bf q}^{\vee}_{2})^t$ occurs with an even multiplicity, so  $ {\bf q}^{\vee}_{2}$ consists of even parts. 
 From Lemma \ref{lem:collapse-even-shift}, we must have $$\left(({\bf q}^{\vee}_{1}\csqcup [1])\csqcup {\bf q}^{\vee}_{2}\csqcup \left(\csqcup_{i}2{\bf p}_{i}\right)\right)_B=\left(({\bf q}^{\vee}_{1}\csqcup [1])_B\csqcup {\bf q}^{\vee}_{2}\csqcup \left(\csqcup_{i}2{\bf p}_{i}\right)\right)_B.$$
 Also we know that
$({\bf q}^{\vee}_{1}\csqcup [1])_B={\bf p}_0$ and ${\bf q}^{\vee}_{2}={\bf p}_{\frac{1}{2}}$. Thus we have
$$\mathcal{O}_{{\rm Ann}(L(\lambda))}=\left(\mathbf{p}_0  \overset{c}{\sqcup}  \mathbf{p}_{\frac{1}{2}}  \overset{c}{\sqcup} \left( \overset{c}{\sqcup_i}  2\mathbf{p}_i \right)\right)_B.$$

This completes the proof for this case.
\item When  $({\bf q}^{\vee}_{2})^t$  contains  even parts  with odd multiplicity, let $\mathbf{v}$ denote the partition formed by the largest such even parts of $({\bf q}^{\vee}_{2})^t$, each appearing exactly once. Let $\mathbf{h}$ denote the remaining parts in $({\bf q}^{\vee}_{1})^t\cup ({\bf q}^{\vee}_{2})^t\cup(\bigcup\limits_{1\leq i\leq k}({\bf p}_{i}^t\cup {\bf p}_{i}^t))$.

We know that $\csqcup_i (2{\bf p}_{i})$ contains only even parts. Thus each column of it occurs with an even multiplicity. 
From the construction of $\mathbf{v}$, we know that ${\bf q}^{\vee}_{2}\csqcup (-\mathbf{v}^t)$ consists of odd columns with even multiplicity and even columns with even multiplicity.
Thus each column of it occurs with an even multiplicity.

Then from Corollary \ref{type-b-som}, we have
\begin{align*}
    \mathcal{O}_{{\rm Ann}(L(\lambda))} &= d_{\rm Som}^{C}(L^{\vee}, \mathcal{O}_{L^\vee})=d_{\text{Som}}(\mathbf{v},\mathbf{h})=(\mathbf{v} \cup \mathbf{h}^B)^t_B\\
    &=(\mathbf{v}^t \csqcup (\mathbf{h}^t)^+_B)_B\quad \quad\quad\text{by Lemma \ref{pm-col}}
    \\
    &=(\mathbf{v}^t \csqcup (\mathbf{h}^t\csqcup [1])_B)_B\\
    &=\left(\mathbf{v}^t \csqcup \left({\bf q}^{\vee}_{1}\csqcup {\bf q}^{\vee}_{2}\csqcup (-\mathbf{v}^t)\csqcup(\csqcup_i 2{\bf p}_{i})\csqcup [1]\right)_B   \right)_B\\
=&\left(\mathbf{v}^t  \csqcup \left(({\bf q}^{\vee}_{1}\csqcup [1])_B\csqcup {\bf q}^{\vee}_{2}\csqcup (-\mathbf{v}^t)\csqcup(\csqcup_i 2{\bf p}_{i})\right)_B \right)_B \\  
& \quad \quad \quad \quad \quad \quad \quad \quad \quad \quad \quad  \text{by Lemma \ref{lem:collapse-even-shift}}\\
=&\left(\mathbf{v}^t  \csqcup ({\bf q}^{\vee}_{1}\csqcup [1])_B\csqcup {\bf q}^{\vee}_{2}\csqcup (-\mathbf{v}^t)\csqcup(\csqcup_i 2{\bf p}_{i}) \right)_B \\
& \quad \quad \quad \quad \quad \quad \quad \quad \quad \quad \quad  \text{by Lemma \ref{lem:2}}\\
   =&\left(({\bf q}^{\vee}_{1}\csqcup [1])_B\csqcup {\bf q}^{\vee}_{2}\csqcup(\csqcup_i 2{\bf p}_{i}) \right)_B \\
    =&\left(\mathbf{p}_0  \overset{c}{\sqcup}  \mathbf{p}_{\frac{1}{2}}  \overset{c}{\sqcup} \left( \overset{c}{\sqcup_i}  2\mathbf{p}_i \right)\right)_B.
\end{align*}

This completes the proof for this case.
\end{enumerate}

 \subsection{Proof for type $C_n$}

Let $L(\lam)$ be a highest weight module of $\mathfrak{sp}(2n,\mathbb{C})$ and
		$[\lambda]=(\lambda)_{(0)} \cup (\lambda)_{(\frac{1}{2})}\cup [\lambda]_3$
				with $[\lambda]_3=\{{\lambda}_{{Y}_1},\dots,{\lambda}_{{Y}_m}\}$.
    Denote
 $H:=\Phi_{[\lambda]}$.

For $x= (\lambda)_{(0)}$,
  we have two special partitions ${\bf p}_0={\bf q}_{1}:=p_{C}(x^-)^s$ of type $C$ and  ${\bf q}_{1}^B={\bf q}^{\vee}_{1}:=p_{B}(x^-)^s$ of type $B$. For $y=(\lambda)_{(\frac{1}{2})}$,
  we have a special partition ${\bf q}_{2}={\bf q}^{\vee}_{2}:=p_{D}(y^-)^s$ of type $D$. For $\lambda_{Y_i}\in [\lambda]_3$, we have a special partition ${\bf p}_{i}:=p(\tilde{\lambda}_{Y_i})$ of type $A$.

Then we have 
   $$\pi_{w_{\lam}}={\bf q}_{1}\times {\bf q}_{2}\times \prod\limits_{1\leq i\leq k}{\bf p}_{i} \in {\rm Irr}(W(H)), $$ $$\pi^{\vee}_{w_{\lam}}={\bf q}^{\vee}_1\times {\bf q}_2\times \prod\limits_{1\leq i\leq k}{\bf p}_{i} \in {\rm Irr}(W(H^{\vee})), $$ 
   and 
$$d_{\rm LS}^{H^\vee}(\mathcal{O}(\pi^{\vee}_{w_{\lam}}))=({\bf q}^{\vee}_{1})^t\times ({\bf q}_{2}^t)_D\times \prod\limits_{1\leq i\leq k}{\bf p}_{i}^t = {\rm sat}_{L^\vee}^{H^\vee} (\mathcal{O}_{L^\vee}),$$
where $L^\vee \subset \Phi^\vee$ and $\mathcal{O}_{L^\vee}$ is a distinguished orbit of $L^\vee$.

Now we give our proof for the type $C$.

\begin{enumerate}
   \item When  $({\bf q}_{2}^t)_D$ contains no odd parts  with odd multiplicity, let $\mathbf{v}=\varnothing$ and $\mathbf{h}=({\bf q}^{\vee}_{1})^t\cup ({\bf q}_{2}^t)_D\cup(\bigcup\limits_{1\leq i\leq k}({\bf p}_{i}^t\cup {\bf p}_{i}^t))$.  Now $({\bf q}_{2}^t)_D$ is a special partition of type $D$ and each part of it has an even multiplicity, so  $(({\bf q}_{2}^t)_D)^t$ is a partition of type $C$ and each part of it is even.
   
    Then from Corollary \ref{type-c-som}, we have
\begin{align*}\mathcal{O}_{{\rm Ann}(L(\lambda))} &= d_{\rm Som}^{B}(L^{\vee}, \mathcal{O}_{L^\vee})=d_{\text{Som}}(\mathbf{v},\mathbf{h})=(\mathbf{v} \cup \mathbf{h}^C)^t_C=(\mathbf{h}^C)^t_C\\
&=((\mathbf{h}^{-})_{C})^{t}_C\\
&=((\mathbf{h}^{t-})_{C})_C \quad \quad \quad \text{by Lemma \ref{pm-col}}\\
&=(\mathbf{h}^{t-})_{C}\\
&=(\mathbf{h}^{t})^{C}\\
&=\left({\bf q}^{\vee}_{1}\csqcup (({\bf q}_{2}^t)_D)^t\csqcup \left(\csqcup_{i}2{\bf p}_{i}\right)\right)^C\\
&=\left(({\bf q}^{\vee}_{1})^C\csqcup (({\bf q}_{2}^t)_D)^t\csqcup \left(\csqcup_{i}2{\bf p}_{i}\right)\right)_C \quad \text{by Lemma \ref{C-empty}}\\
&=\left(\mathbf{p}_0  \overset{c}{\sqcup}  \mathbf{p}_{\frac{1}{2}}  \overset{c}{\sqcup} \left( \overset{c}{\sqcup_i}  2\mathbf{p}_i \right)\right)_C \quad \text{since $({\bf q}^{\vee}_{1})^C=({\bf q}^{B}_{1})^C={\bf q}_1={\bf p}_0$}.
\end{align*}
This completes the proof for this case.
    
    \item When   $({\bf q}_{2}^t)_D$  contains odd parts  with odd multiplicity, let $\mathbf{v}$ denote the partition formed by the maximal such odd parts of  $({\bf q}_{2}^t)_D$, each appearing exactly once. Let $\mathbf{h}$ denote the remaining parts in $({\bf q}^{\vee}_{1})^t\cup ({\bf q}_{2}^t)_D\cup(\bigcup\limits_{1\leq i\leq k}({\bf p}_{i}^t\cup {\bf p}_{i}^t))$.
    Now $({\bf q}_{2}^t)_D\cup (-\mathbf{v})$  is a special partition of type $D$ and each part of it has an even multiplicity, so  $(({\bf q}_{2}^t)_D)^t\csqcup (-\mathbf{v}^t)$ is a partition  with the property that every column
occurs with an even multiplicity.

Then from Corollary \ref{type-c-som},  we have
\begin{align*}
    \mathcal{O}_{{\rm Ann}(L(\lambda))} &= d_{\rm Som}^{B}(L^{\vee}, \mathcal{O}_{L^\vee})=d_{\text{Som}}(\mathbf{v},\mathbf{h})\\&=(\mathbf{v} \cup \mathbf{h}^C)^t_C=(\mathbf{v}^t \csqcup (\mathbf{h}^C)^t)_C
    \\
    &=(\mathbf{v}^t \csqcup ((\mathbf{h}^-)_C)^t)_C
    \\
    &=(\mathbf{v}^t \csqcup (\mathbf{h}^{t-})_C)_C\quad \quad\quad\quad\quad\quad \text{by Lemma \ref{pm-col}}
    \\
    &=\left(\mathbf{v}^t \csqcup \left({\bf q}^{\vee}_{1}\csqcup (({\bf q}_{2}^t)_D)^t\csqcup (-\mathbf{v}^t)\csqcup(\csqcup_i 2{\bf p}_{i})\right)^C   \right)_C\\
    &=\left(\mathbf{v}^t \csqcup \left(({\bf q}^{\vee}_{1})^C\csqcup (({\bf q}_{2}^t)_D)^t\csqcup (-\mathbf{v}^t)\csqcup(\csqcup_i 2{\bf p}_{i})\right)_C   \right)_C\\
    & \qquad\qquad\qquad\qquad\qquad\qquad\qquad \text{by Lemma \ref{C-empty}}\\
    &=\left(\mathbf{v}^t \csqcup ({\bf q}^{\vee}_{1})^C\csqcup (({\bf q}_{2}^t)_D)^t\csqcup (-\mathbf{v}^t)\csqcup(\csqcup_i 2{\bf p}_{i})\right)_C \\
    & \quad \quad \quad \quad \quad \quad \quad \quad \quad \quad \quad \quad\quad  \quad   \text{by Lemma \ref{lem-C}}\\
    &=\left( ({\bf q}^{\vee}_{1})^C\csqcup (({\bf q}_{2}^t)_D)^t\csqcup (\csqcup_i 2{\bf p}_{i})\right)_C\\
    &=\left(\mathbf{p}_0  \overset{c}{\sqcup}  \mathbf{p}_{\frac{1}{2}}  \overset{c}{\sqcup} \left( \overset{c}{\sqcup_i}  2\mathbf{p}_i \right)\right)_C \quad \quad   \text{since $({\bf q}^{\vee}_{1})^C=({\bf q}^{B}_{1})^C={\bf q}_1={\bf p}_0$}.
\end{align*}

   This completes the proof for this case.
\end{enumerate}

 \subsection{Proof for type $D_n$}
Let $L(\lam)$ be a highest weight module of $\mathfrak{so}(2n,\mathbb{C})$ and
		$[\lambda]=(\lambda)_{(0)} \cup (\lambda)_{(\frac{1}{2})}\cup [\lambda]_3$
				with $[\lambda]_3=\{{\lambda}_{{Y}_1},\dots,{\lambda}_{{Y}_m}\}$.
    Denote
 $H:=\Phi_{[\lambda]}$.

Denote $x_1= (\lambda)_{(0)}$ and $ x_2=(\lambda)_{(\frac{1}{2})}$. For each $1\leq i\leq 2$,
  we have a special partition ${\bf q}^{\vee}_{i}={\bf q}_{i}:=p_{D}(x_i^-)^s$ of type $D$. For $\lambda_{Y_i}\in [\lambda]_3$, we have a special partition ${\bf p}_{i}:=p(\tilde{\lambda}_{Y_i})$ of type $A$.
  
Then we have 
   $$\pi_{w_{\lam}}=\pi^{\vee}_{w_{\lam}}={\bf q}_{1}\times {\bf q}_{2}\times \prod\limits_{1\leq i\leq k}{\bf p}_{i} \in {\rm Irr}(W(H)) $$ 
   and
$$d_{\rm LS}^{H^\vee}(\mathcal{O}(\pi^{\vee}_{w_{\lam}}))=({\bf q}_{1}^t)_D\times ({\bf q}_{2}^t)_D\times \prod\limits_{1\leq i\leq k}{\bf p}_{i}^t = {\rm sat}_{L^\vee}^{H^\vee} (\mathcal{O}_{L^\vee}),$$
where $L^\vee \subset \Phi^\vee$ and $\mathcal{O}_{L^\vee}$ is a distinguished orbit of $L^\vee$.

Now we give our proof for the type $D$.

\begin{enumerate}
    \item When  $({\bf q}_{2}^t)_D$ contains no odd parts  with odd multiplicity, let $\mathbf{v}=\varnothing$ and $\mathbf{h}=({\bf q}_{1}^t)_D\cup ({\bf q}_{2}^t)_D\cup(\bigcup\limits_{1\leq i\leq k}({\bf p}_{i}^t\cup {\bf p}_{i}^t))$.
    Now $({\bf q}_{2}^t)_D$ is a special partition of type $D$ and each part of it has an even multiplicity, so  $(({\bf q}_{2}^t)_D)^t$ is a partition of type $C$ and each part of it is even.

    Then from Corollary \ref{type-D-som}, we have
\begin{align*}\mathcal{O}_{{\rm Ann}(L(\lambda))} &= d_{\rm Som}^{D}(L^{\vee}, \mathcal{O}_{L^\vee})=d_{\text{Som}}(\mathbf{v},\mathbf{h})=(\mathbf{v} \cup (\mathbf{h}^t_D)^t)^t_D\\
&=\mathbf{h}^t_D\\
&=\left(  ({\bf q}_{1}^t)_D\cup ({\bf q}_{2}^t)_D\cup(\bigcup\limits_{1\leq i\leq k}({\bf p}_{i}^t\cup {\bf p}_{i}^t)) \right)^t_D\\
&=\left((({\bf q}_{1}^t)_D)^t\csqcup (({\bf q}_{2}^t)_D)^t\csqcup \left(\csqcup_{i}2{\bf p}_{i}\right)\right)_D\\
&=\left(((({\bf q}_{1}^t)_D)^t)_D\csqcup (({\bf q}_{2}^t)_D)^t\csqcup \left(\csqcup_{i}2{\bf p}_{i}\right)\right)_D\\
& \quad \quad  \quad \quad \quad \quad \quad \quad \quad  \text{by Lemma \ref{lem:collapse-even-shift}}\\
 &=\left(\mathbf{p}_0  \overset{c}{\sqcup}  \mathbf{p}_{\frac{1}{2}}  \overset{c}{\sqcup} \left( \csqcup_i  2{\bf p}_i \right)\right)_D,
\end{align*}
where the last identity follows from \cite[Cor. 6.3.6]{CM} and the fact that  ${\bf q}_{1}={\bf p}_0$ is special.

This completes the proof for this case.
    
    \item When  $({\bf q}_{2}^t)_D$ contains  odd parts  with odd multiplicity, let $\mathbf{v}$ denote the partition formed by the maximal such odd parts of  $({\bf q}_{2}^t)_D$, each appearing exactly once. Let $\mathbf{h}$ denote the remaining parts in $({\bf q}_{1}^t)_D\cup ({\bf q}_{2}^t)_D\cup(\bigcup\limits_{1\leq i\leq k}({\bf p}_{i}^t\cup {\bf p}_{i}^t))$.
Now $({\bf q}_{2}^t)_D\cup (-\mathbf{v})$  is a special partition of type $D$ and each part of it has an even multiplicity, so  $(({\bf q}_{2}^t)_D)^t\csqcup (-\mathbf{v}^t)$ is a partition  with the property that every column
occurs with an even multiplicity.

   Then from Corollary \ref{type-D-som}, we have
\begin{align*}
    \mathcal{O}_{{\rm Ann}(L(\lambda))} &= d_{\rm Som}^{D}(L^{\vee}, \mathcal{O}_{L^\vee})=d_{\text{Som}}(\mathbf{v},\mathbf{h})=(\mathbf{v} \cup (\mathbf{h}^t_D)^t)^t_D\\
    &=(\mathbf{v}^t\csqcup \mathbf{h}^t_D)_D\\
    &=\left(\mathbf{v}^t \csqcup \left((({\bf q}_{1}^t)_D)^t\csqcup (({\bf q}_{2}^t)_D)^t\csqcup (-\mathbf{v}^t)\csqcup(\csqcup_i 2{\bf p}_{i})\right)_D   \right)_D\\
    &=\left(\mathbf{v}^t \csqcup \left(((({\bf q}_{1}^t)_D)^t)_D\csqcup (({\bf q}_{2}^t)_D)^t\csqcup (-\mathbf{v}^t)\csqcup(\csqcup_i 2{\bf p}_{i})\right)_D   \right)_D\\
    & \quad \quad\quad\quad\quad\quad\quad \quad\quad \text{by Lemma \ref{lem:collapse-even-shift}}\\
    &=\left(\mathbf{v}^t \csqcup \left({\bf q}_{1}\csqcup (({\bf q}_{2}^t)_D)^t\csqcup (-\mathbf{v}^t)\csqcup(\csqcup_i 2{\bf p}_{i})\right)_D   \right)_D\\
    & \quad \quad\quad \quad\quad \quad\quad\quad\quad\text{$((({\bf q}_{1}^t)_D)^t)_D={\bf q}_{1}$ since ${\bf q}_{1}$ is a special partition of type $D$}\\
    &=\left(\mathbf{v}^t \csqcup {\bf q}_{1}\csqcup (({\bf q}_{2}^t)_D)^t\csqcup (-\mathbf{v}^t)\csqcup(\csqcup_i 2{\bf p}_{i})   \right)_D\\
    & \quad \quad\quad \quad\quad\quad\quad\quad\quad \text{by Lemma \ref{lem-C}}\\
    &=\left( {\bf q}_{1}\csqcup (({\bf q}_{2}^t)_D)^t\csqcup (\csqcup_i 2{\bf p}_{i})   \right)_D\\
    &=\left(\mathbf{p}_0  \overset{c}{\sqcup}  \mathbf{p}_{\frac{1}{2}}  \overset{c}{\sqcup} \left( \csqcup_i  2{\bf p}_i \right)\right)_D.
\end{align*}

This completes the proof for this case.

\end{enumerate}

\subsection{The revised partition algorithm}

From the symmetry of the extended Dynkin diagrams of types $C$ and $D$, together with the new proof of the partition algorithm  in \S \ref{new-proof}, we can obtain the following result.

\begin{Cor}\label{imp-partition}
  Suppose $\mathfrak{g} = \mathfrak{so}(2n+1, \mathbb{C}), \mathfrak{sp}(n, \mathbb{C})$ or $\mathfrak{so}(2n, \mathbb{C})$, $\lambda\in \mathfrak{h}^*$ and
		$[\lambda]=(\lambda)_{(0)} \cup (\lambda)_{(\frac{1}{2})}\cup [\lambda]_3$
		with $[\lambda]_3=\{{\lambda}_{{Y}_1},\dots,{\lambda}_{{Y}_m}\}$. If $|p((\lambda)_{(0)}^-)|
        \geq |p((\lambda)_{(\frac{1}{2})}^-)|$, let
		\begin{enumerate}
			\item $\mathbf{p}_{0}$ be the $X$-type special partition associated with $(\lambda)_{(0)}$;
			\item ${\mathbf p}_{\frac{1}{2}}$ be the $C$-type special  partition (for type $B$) or $C$-type metaplectic special partition (for types $C$ and $D$) associated with $(\lambda)_{(\frac{1}{2})}$;
			\item $\mathbf{p}_{i}$ be the $A$-type partition associated with $\tilde{\lambda}_{Y_i}$.
		\end{enumerate}
	If $|p((\lambda)_{(0)}^-)|<|p((\lambda)_{(\frac{1}{2})}^-)|$, we can swap the roles of $\mathbf{p}_{0}$ and  ${\mathbf p}_{\frac{1}{2}}$ for types $B$ and $D$. For type $C$, we still use the previous choice.
        
        Let $\mathbf{p}_{\lambda}$ be the $X$-collapse of
 \[
\mathbf{p}_0  \overset{c}{\sqcup}  \mathbf{p}_{\frac{1}{2}}  \overset{c}{\sqcup} \left( \overset{c}{\sqcup_i}  2\mathbf{p}_i \right)
\]
		Then we have
		\[
		V(\Ann (L(\lambda)))=\overline{\mathcal{O}}_{\mathbf{p}_{\lambda}}.
		\]
 When $n$ is even and ${\mathcal{O}}_{\mathbf{p}_{\lambda}}$ is a very even orbit,
 its numeral is determined by the number $k(\lambda)$ of very even orbits with numeral II in the set of very even orbits of type $D$:    $\{\mathcal{O}_{\bfpp_{0}}, \mathcal{O}_ {{\bf p}_{\frac{1}{2}}},  \mathcal{O}_{2\bfpp_i}| 1\leq i\leq m  \}$. 
So $\mathcal{O}_{{\bf p}_{\lambda}}$ will be of type I if $k(\lambda)\equiv 0 
  ~(\mathrm{mod}~ {2})$ and of type II if $k(\lambda)\equiv 1 
  ~(\mathrm{mod}~ {2})$.
\end{Cor}

\begin{ex}\label{d-ex}
    Let $\mathfrak{g}=\mathfrak{so}(27, \mathbb{C})$.
		Suppose $$\lambda=(2.5,3.5,-4.5,-6.5,-8.5,-10.5,4,2,3,1.1,3.1,4.1,0.9).$$ Then we have $$\lambda_{(0)}=(4,2,3),$$ 
        $$\lambda_{(\frac{1}{2})}=(2.5,3.5,-4.5,-6.5,-8.5,-10.5)$$
        and $$\lambda_{Y_1}=(1.1,3.1,4.1,0.9).$$

For $\lambda_{(0)}$, from the Young tableau $P(\lambda_{(0)}^-)$,  we can obtain 
        $p(\lambda_{(0)}^-)=[2,2,1,1]$, which yields a special partition of type $B$: ${\bf p}_0=[3,1^4]$.

For $\lambda_{(\frac{1}{2})}$, from the Young tableau $P(\lambda_{(\frac{1}{2})}^-)$,  we can obtain 
        $p(\lambda_{(\frac{1}{2})}^-)=[3,3,2^3]$, which yields a special partition of type $C$: ${\bf p}_{\frac{1}{2}}=[3,3,2^3]$.

		For $\lambda_{Y_1}$, by using the RS algorithm for $\tilde{\lambda}_{Y_1}$,  we can obtain: ${\bf p}_1=[3,1]$.

From Theorem \ref{pa-suanfa}, we have $V(\Ann (L(\lambda)))=\overline{\mathcal{O}}_{[13,7,2,2,1,1]}$ since
    $${\bf p}_{\lambda}=(\mathbf{p}_0  \overset{c}{\sqcup}  \mathbf{p}_{\frac{1}{2}}  \overset{c}{\sqcup}   2\mathbf{p}_1 )_B=([3,1^4]\csqcup [3,3,2^3]\csqcup 2[3,1])_B=[12,6,3^3]_B=[11,7,3^3].$$
        
Now we swap the roles of $\mathbf{p}_0$  and  $\mathbf{p}_{\frac{1}{2}}$. 		From $p(\lambda_{(\frac{1}{2})}^-)=[3,3,2^3]$, we can obtain a special partition of type $B$: ${\bf p}_0=[3^3,2,2]$.
From $p(\lambda_{0}^-)=[2,2,1,1]$, we can obtain a special partition of type $C$:  ${\bf p}_{\frac{1}{2}}=[2,2,1,1]$.
We still have ${\bf p}_1=[3,1]$.
		 Then from Corollary \ref{imp-partition}, 
    we still have $V(\Ann (L(\lambda)))=\overline{\mathcal{O}}_{[13,7,2,2,1,1]}$ since
    $${\bf p}_{\lambda}=(\mathbf{p}_0  \overset{c}{\sqcup}  \mathbf{p}_{\frac{1}{2}}  \overset{c}{\sqcup}   2\mathbf{p}_1 )_B=([3^3,2,2]\csqcup [2,2,1,1]\csqcup 2[3,1])_B=[11,7,4,3,2]_B=[11,7,3,3,3]. $$
    \end{ex}

\begin{ex}
       Let $\mathfrak{g}=\mathfrak{so}(26, \mathbb{C})$.
		Suppose $$\lambda=(1.5,3.5,-4.5,-6.5,-8.5,-10.5,1,2,3,1.1,3.1,4.1,0.9).$$ Then we have $$\lambda_{(0)}=(1,2,3),$$ 
        $$\lambda_{(\frac{1}{2})}=(1.5,3.5,-4.5,-6.5,-8.5,-10.5)$$
        and $$\lambda_{Y_1}=(1.1,3.1,4.1,0.9).$$

For $\lambda_{(0)}$, from the Young tableau $P(\lambda_{(0)}^-)$,  we can obtain 
        $p(\lambda_{(0)}^-)=[3,3]$, which yields a special partition of type $D$: ${\bf p}_0=[3,3]$.

For $\lambda_{(\frac{1}{2})}$, from the Young tableau $P(\lambda_{(\frac{1}{2})}^-)$,  we can obtain 
        $p(\lambda_{(\frac{1}{2})}^-)=[3,3,2^3]$, which yields a special partition of type $D$: ${\bf q}_2=[3,3,2,2,1,1]$, and a  $C$-type metaplectic special partition: ${\bf p}_{\frac{1}{2}}=(({\bf q}_2^t)_D)^t=(([3,3,2,2,1,1]^t)_D)^t=[4,2^4]$.

		For $\lambda_{Y_1}$, by using the RS algorithm for $\tilde{\lambda}_{Y_1}$,  we can obtain: ${\bf p}_1=[3,1]$.

From Theorem \ref{pa-suanfa}, 
    we have $V(\Ann (L(\lambda)))=\overline{\mathcal{O}}_{[13,7,2,2,1,1]}$ since
    $${\bf p}_{\lambda}=(\mathbf{p}_0  \overset{c}{\sqcup}  \mathbf{p}_{\frac{1}{2}}  \overset{c}{\sqcup}   2\mathbf{p}_1 )_D=([3,3]\csqcup [4,2^4]\csqcup 2[3,1])_D=[13,7,2,2,2]_D=[13,7,2,2,1,1]. $$
        
Now we swap the roles of $\mathbf{p}_0$  and  $\mathbf{p}_{\frac{1}{2}}$. 		From
       $p(\lambda_{(\frac{1}{2})}^-)=[3,3,2^3]$, we can obtain a special partition of type $D$: ${\bf p}_0=[3,3,2,2,1,1]$.
From 
        $p(\lambda_{0}^-)=[3,3]$, we can obtain a special partition of type $D$: ${\bf q}_2=[3,3]$, and a $C$-type metaplectic special partition: ${\bf p}_{\frac{1}{2}}=(({\bf q}_2^t)_D)^t=(([3,3]^t)_D)^t=[4,2]$.
We still have ${\bf p}_1=[3,1]$.
		 Then from Corollary \ref{imp-partition}, 
    we still have $V(\Ann (L(\lambda)))=\overline{\mathcal{O}}_{[13,7,2,2,1,1]}$ since
    $${\bf p}_{\lambda}=(\mathbf{p}_0  \overset{c}{\sqcup}  \mathbf{p}_{\frac{1}{2}}  \overset{c}{\sqcup}   2\mathbf{p}_1 )_D=([3,3,2,2,1,1]\csqcup [4,2]\csqcup 2[3,1])_D=[13,7,2,2,1,1]_D=[13,7,2,2,1,1]. $$
    \end{ex}

\begin{Rem}
   The algorithm in Corollary \ref{imp-partition} is called \emph{the revised partition algorithm}. For types $B$ and $C$, we can still use the original partition algorithm in Theorem \ref{recipe}. But for type $D$, we may use the revised partition algorithm when $|p((\lambda)_{(0)}^-)|<|p((\lambda)_{(\frac{1}{2})}^-)|$. 
\end{Rem}

\subsection*{Acknowledgments} 
Y. Luan is supported by Shandong Provincial Natural Science Foundation (Grant No. ZR2025MS36), Guangdong Basic and Applied Basic Research Foundation (Grant No. 2023A1515110189). The authors thank Fan Gao  for very helpful discussions about the Sommers duality.

\bibliographystyle{plain}
\bibliography{BJL}
\end{document}